\documentclass{article}

\usepackage[utf8]{inputenc}
\usepackage[T1]{fontenc}
\usepackage[french,english]{babel}
\usepackage[a4paper]{geometry}
\usepackage{lmodern}
\usepackage{microtype}
\usepackage[dvipsnames]{xcolor}

\usepackage{amsmath,amsfonts,amssymb,amsthm,amscd}
\usepackage{mathabx}
\usepackage{bbm}
\usepackage{stmaryrd}
\usepackage{thmtools}

\usepackage[colorlinks=true,citecolor=ForestGreen,linkcolor=RoyalBlue,urlcolor=Bittersweet]{hyperref}

\usepackage{enumerate}
\usepackage{enumitem}

\newcommand{\set}[1]{\left\lbrace #1 \right\rbrace}
\newcommand{\lrpar}[1]{\left( #1 \right)}
\newcommand{\lrangle}[1]{\left< #1 \right>}

\newcommand{\llrrbracket}[1]{\left\llbracket #1 \right\rrbracket}

\newcommand{\seq}[2]{\left(#1\right)_{#2}}
\newcommand{\floor}[1]{\left\lfloor #1 \right\rfloor}
\newcommand{\abs}[1]{\ensuremath{\left| #1 \right|}}
\newcommand{\norm}[1]{\ensuremath{\left\| #1 \right\|}}

\newcommand{\proba}[1]{\ensuremath{\mathbb P\left(#1\right)}}

\newcommand{\esp}[1]{\ensuremath{\mathbb E\left[#1\right]}}

\newcommand{\indset}[1]{\ensuremath{\mathbbm{1}_{{#1}}}}
\newcommand{\ind}[1]{\indset{\set{#1}}}

\renewcommand{\epsilon}{\varepsilon}
\renewcommand{\phi}{\varphi}

\renewcommand{\bar}{\overline}

\DeclareMathOperator{\id}{id}
\DeclareMathOperator{\Tr}{Tr}
\DeclareMathOperator{\Lip}{Lip}

\let\div\undefined
\DeclareMathOperator{\div}{div}

\newcommand{\M}{{\mathcal M}}
\renewcommand{\L}{{\mathcal L}}
\newcommand{\C}{{\mathcal C}}
\newcommand{\T}{{\mathbb T}}
\newcommand{\R}{{\mathbb R}}
\newcommand{\E}{{\mathbb E}}

\newcommand{\fSet}{{L^2(\M^{-1})}}

\newcommand{\edstopped}{^{\epsilon,\delta,\tau^\delta}}
\newcommand{\edstoppedi}{^{\epsilon_i,\delta_i,\tau^{\delta_i}}}

\newcommand{\dstopped}{^{\delta,\tau^\delta}}
\newcommand{\dstoppedi}{^{\delta_i,\tau^{\delta_i}}}

\newcommand{\n}{\ell}

\theoremstyle{plain}
	\newtheorem{theorem}{Theorem}[section]

	\newtheorem{proposition}[theorem]{Proposition}
	
	\newtheorem{lemma}[theorem]{Lemma}

\theoremstyle{definition}
	\newtheorem{definition}{Definition}[section]
	\newtheorem{assumption}{Assumption}
\theoremstyle{remark}
	
	\newtheorem{example}{Example}[section]

\usepackage{authblk}

\title{Asymptotic behavior of a class of multiple time scales stochastic kinetic equations}

\author{Charles-Edouard Br\'ehier and Shmuel Rakotonirina-Ricquebourg\thanks{Univ Lyon, Université Claude Bernard Lyon 1, CNRS UMR 5208, Institut Camille Jordan, 43 blvd. du 11 novembre 1918, F-69622 Villeurbanne cedex, France,\texttt{brehier@math.univ-lyon1.fr},\texttt{rakoto@math.univ-lyon1.fr}}}

\begin{document}

\maketitle

\begin{abstract}
We consider a class of stochastic kinetic equations, depending on two time scale separation parameters $\epsilon$ and $\delta$: the evolution equation contains singular terms with respect to $\epsilon$, and is driven by a fast ergodic process which evolves at the time scale $t/\delta^2$. We prove that when $(\epsilon,\delta)\to (0,0)$ the density converges to the solution of a linear diffusion PDE. This is a mixture of diffusion approximation in the PDE sense (with respect to the parameter $\epsilon$) and of averaging in the probabilistic sense (with respect to the parameter $\delta$). The proof employs stopping times arguments and a suitable perturbed test functions approach which is adapted to consider the general regime $\epsilon\neq \delta$.
\end{abstract}


\section{Introduction}

Multiscale and/or stochastic models are popular in all fields of science and engineering. In this paper, we consider a stochastic kinetic partial differential equation of the type
\begin{equation} \label{eq:SPDEintro}
	\partial_t f^{\epsilon,\delta} + \frac{1}{\epsilon} a(v) \cdot \nabla_x f^{\epsilon,\delta} + b(v) \cdot \nabla_x f^{\epsilon,\delta} + \sigma(m^\delta(t,x)) f^{\epsilon,\delta} = \frac{1}{\epsilon^2} Lf^{\epsilon,\delta},
\end{equation}
with initial condition $f^{\epsilon,\delta}(0) = f^{\epsilon,\delta}_0$. The unknow $f^{\epsilon,\delta}$ is a function of time $t\ge 0$, position $x\in\T^d$ (the flat $d$-dimensional torus) and velocity $v\in V$. Assumptions on the velocity fields $a$ and $b$ and on the mapping $\sigma$ are given below (see Section~\ref{sec:setting}). Note that $f^{\epsilon,\delta}(t,x,v)$ may be interpreted as a density of particles with position $x$ and velocity $v$ at time $t$; the system is not conservative (the integral of $f^{\epsilon,\delta}$ is not constant) due to the source term $\sigma(m^\delta) f^{\epsilon,\delta}$. In addition, the linear operator $L$ describes interactions between the particles: in this paper, we assume that $L$ is the Bhatnagar-Gross-Krook operator, given by
\begin{equation*}
	Lf = \rho \mathcal M - f,
\end{equation*}
where $\mu$ is a $\sigma$-finite measure on $V$, the spatial density is defined by $\rho \doteq \lrangle f \doteq \int_V f d\mu$, and where $\mathcal M \in L^1(V,d\mu)$ is a density function, often called the Maxwellian (see Assumption~\ref{ass:coefficients}).

The evolution equation~\eqref{eq:SPDEintro} depends on the so-called driving process $m^\delta$, defined as follows: one has $m^\delta(t,x)=m(t/\delta^2,x)$.

The stochastic evolution equation~\eqref{eq:SPDEintro} depends on two parameters $\epsilon$ and $\delta$. In this paper, we are interested in the asymptotic behavior when $(\epsilon,\delta)\to (0,0)$. We prove the following result: $\rho^{\epsilon,\delta}=\lrangle{f^{\epsilon,\delta}}$ converges to the solution $\bar\rho$ of the following partial differential equation:
\begin{equation} \label{eq:limit_equation_intro}
	\partial_t \bar\rho + J\cdot \nabla \bar\rho +\bar \sigma \, \bar\rho = \div (K\nabla \bar\rho),
\end{equation}
where $J$, $\bar\sigma$ and $K$ are defined below, with initial condition $\bar\rho(0)=\bar\rho_0=\lim \rho^{\epsilon,\delta}(0)$. We refer to the main results of this paper, Theorems~\ref{thm:main_Hneg} and~\ref{thm:main_L2} for rigorous statements, in particular concerning the mode of convergence. Let us describe how the form of the limit equation~\eqref{eq:limit_equation_intro} arises in the asymptotic regime.

On the one hand, the parameter $\epsilon$ drives the behavior of the deterministic part of the evolution. Under appropriate assumptions (including a centering condition for the velocity field $a$), the spatial density $\rho^{\epsilon,\delta}=\lrangle f^{\epsilon,\delta}$ converges when $\epsilon\to 0$ to the solution of a diffusion partial differential equation, where the velocity variable $v$ has been eliminated. In the literature, such convergence results are referred to as diffusion approximation results, see for instance~\cite{degond2000diffusion}. The result is also partly an averaging result, since the term $b(v)\cdot\nabla_x$ is replaced in the limit equation by $J\cdot\nabla_x$, where $J$ is the average of $b$ (with respect to an appropriate measure).

On the other hand, the parameter $\delta$ is a time-scale separation parameter which determines the random part of the evolution. When the process $m$ such that $m^\delta(t,x)=m(t/\delta^2,x)$ is assumed to be ergodic (see Section~\ref{sec:setting}, for instance one may consider an Ornstein-Uhlenbeck process), the randomness may be eliminated when $\delta\to 0$: only the average $\bar\sigma$ of $\sigma(m)$ with respect to the invariant distribution remains in the limit evolution equation (it is a law of large numbers effect). In the literature, such convergence results are referred to as averaging principle results.

In this paper, we thus prove the mixture of an diffusion approximation result in the PDE sense, and of an averaging principle result in the probability sense, when simultaneously $\epsilon\to 0$ and $\delta \to 0$. To the best of our knowledge, this regime has not been considered in the literature so far. Note that one of the major tasks in the analysis is to consider the general case when $\epsilon$ and $\delta$ go to $0$ independently. Indeed, the analysis would be simpler if $\epsilon=\delta$. Note also that it would be simpler if $\epsilon=0$ and $\delta\to 0$, or if $\epsilon\to 0$ and $\delta=0$, {\it i.e.} if the limits are taken successively. The latter case is not included in the analysis but may be handled with similar techniques. The limit equation~\eqref{eq:limit_equation_intro} is the same in all regimes.

For deterministic problems ($\sigma=0$), diffusion approximation results have been extensively studied. We refer for instance to~\cite{larsen1974asymptotic,bensoussan1979boundary}. Kinetic models with small parameters appear in various situations, for example when studying semi-conductors \cite{golse1992limite} and discrete velocity models \cite{lions1997diffusive}, or as limits for description of systems of particles, either with a single particle \cite{goudon2009stochastic} or multiple particles \cite{poupaud2003classical}. The asymptotic behavior of stochastic kinetic multiscale problems have also been recently studied: we refer to the seminal article~ \cite{debussche2011diffusion}, and the more recent contributions~\cite{debussche2017diffusion,debussche2020diffusion,rakotonirina2020diffusion}. In those works, the authors have obtained diffusion approximation results both in the PDE and the probabilistic senses: the limit equation is a stochastic linear diffusion PDE  driven by a Wiener process (with Stratonovich interpretation). Indeed, in those works $\sigma(m^\delta)$ (with $\delta=\epsilon$) is replaced by $m^{\delta}/\delta$ in~\eqref{eq:SPDEintro}, and the authors assume that the driving process $m^\delta$ satisfies an appropriate centering condition. In the present article, we consider a law of large numbers regime (hence the averaging principle result), instead of a central limit theorem regime. In spite of this fundamental difference, the setting is very close to~\cite{rakotonirina2020diffusion}: in particular a major technical difficulty which is solved in this paper is to avoid boundedness assumptions on the driving process $m$, using only moment conditions, which allow us to encompass for instance Ornstein-Uhlenbeck processes.

The literature concerning the averaging principle for stochastic differential equations and stochastic partial differential equations is huge. The averaging principle in the SDE case has been introduced in the seminal reference ~\cite{khasminskii1968principle}, see also the monograph~\cite{pavliotis2008multiscale} and references therein. In the SPDE case, authors have mainly studied the averaging principle for parabolic semilinear SPDE systems, see for instance~ \cite{cerrai2009khasminskii,cerrai2009averaging,brehier2012strong,brehier2020orders,rockner2020asymptotic} and references therein. Let us also mention the recent preprints~\cite{cerrai2020smoluchowski,xie2021diffusion} where diffusion approximation results are proved for such systems. The list of references above is not exhaustive.

The first main result of this manuscript is Theorem~\ref{thm:main_Hneg}: the convergence of $\rho^{\epsilon,\delta}=\lrangle{f^{\epsilon,\delta}}$ to $\bar\rho$ is understood as convergence in distribution (in the probabilistic sense), in the space $\C^0([0,T],H^{-\varsigma}(\T^d))$, for all arbitrarily small positive $\varsigma$. Under an additional assumption (which allows us to employ an averaging lemma), the convergence holds in the space $L^2([0,T],L^2(\T^d))$, see the second main result of this manuscript, Theorem~\ref{thm:main_L2}. The functional spaces above are the standard spaces where convergence holds in the deterministic case. The convergence in distribution is the natural mode of the convergence for the probabilistic variable. However, since the limit equation~\eqref{eq:limit_equation_intro} is deterministic, if the limit initial condition $\bar\rho_0$ is also deterministic, then the convergence holds in probability.

Let us now describe the main tools for the proof of the main results of this manuscript. We follow a martingale problem approach combined with the perturbed test functions method, as in the classical article \cite{papanicolaou1977martingale} (see also \cite{kushner1984approximation,ethier1986markov,fouque2007wave,pavliotis2008multiscale,de2012diffusion}). Perturbed test functions in the context of PDEs with diffusive limits applies in various situations, for instance in the context of viscosity solutions \cite{evans1989perturbed}, nonlinear Schrödinger equations \cite{de2012diffusion},
a parabolic PDE \cite{pardoux2003homogenization} or, as in this article, kinetic SPDEs \cite{debussche2011diffusion,debussche2017diffusion,debussche2020diffusion,rakotonirina2020diffusion}.

The idea of the perturbed test function is to identify the limit generator $\L$ of the limit equation~\eqref{eq:limit_equation_intro} as
\[
\L\phi=\underset{(\epsilon,\delta) \to (0,0)}\lim \L^{\epsilon,\delta} \phi^{\epsilon,\delta},
\]
where $\phi=\phi(\rho)$ is an arbitrary test function (depending only on the spatial density variable $\rho=\lrangle f$ appearing in the limit equation), and $\phi^{\epsilon,\delta}=\phi^{\epsilon,\delta}(f,m)$ is the perturbed test function given by
\[
\phi^{\epsilon,\delta} = \phi + \epsilon \phi_{1,0} + \epsilon^2 \phi_{2,0} + \delta^2 \phi_{0,2} + \epsilon \delta^2 \phi_{1,2}.
\]
We refer to Proposition~\ref{prop:corrector} for a rigorous statement. Note that due to the assumption that in general $\epsilon\neq\delta$, the construction of the perturbed test function requires to define the corrector $\phi_{1,2}$ (corresponding to the term of the order $\epsilon\delta^2$). This corrector does not appear in the analysis if $\epsilon=\delta$. The construction of $\phi_{1,2}$ is one of the novelties of this manuscript.

In addition, like in the preprint~\cite{rakotonirina2020diffusion}, the driving process is not assumed to be bounded (like in~\cite{debussche2011diffusion} for instance), and only moment conditions are satisfied. This generalization allows us for instance to consider Ornstein-Uhlenbeck processes. The introduction of stopping times arguments is required, hence the need to control the asymptotic behavior of the stopping times $\tau^\delta$ when $\delta\to 0$. Note that we prove below that $\tau^\delta\to\infty$ in probability: the arguments to prove convergence are thus simpler than in~\cite{rakotonirina2020diffusion}.

The main result of this manuscript is the convergence of $\rho^{\epsilon,\delta}$ to $\bar\rho$. In future works, it may be interesting to study the fluctuations, {\it i.e.} to prove that $\rho^{\epsilon,\delta}-\bar\rho$, properly rescaled, converges in distribution to a Gaussian process, solution of a linear stochastic evolution equation. Again the use of the perturbed test functions approach may be a suitable approach. It would also be interesting to investigate rates of convergence, both in the strong and weak senses, in the spirit of~\cite{brehier2012strong,brehier2020orders} concerning parabolic systems, using Kolmogorov equations techniques. Finally, the validity of diffusion approximation and averaging principle results is fundamental for the efficient numerical simulation of the systems. In the deterministic setting, there has been a lot of activity to develop asymptotic preserving and uniformly accurate numerical methods, see  for instance~\cite{jin1999efficient,jin2012asymptotic}. An asymptotic preserving scheme has been proposed in~\cite{ayi2019analysis} for a class of kinetic stochastic equations in the diffusion approximation regime. In a future work~\cite{brehier2021asymptotic}, we plan to investigate the generalization of the asymptotic preserving schemes introduced and analyzed in the recent preprint~\cite{brehier2020asymptotic} for stochastic differential equations.

The manuscript is organized as follows. The setting is described in Section~\ref{sec:setting} (in particular precise assumptions concerning the driving process $m$ are provided). The main results, Theorems~\ref{thm:main_Hneg} and~\ref{thm:main_L2}, are stated and discussed in Section~\ref{sec:results}. Section~\ref{sec:proofdescription} is devoted to the proofs of the main results, using martingale problem formulations, tightness arguments and identification of the limit. Auxiliary fundamental results are stated there: first, Proposition~\ref{prop:taue_to_infty} concerning the asymptotic behavior of the stopping time; second, Proposition~\ref{prop:L2_bound} providing an a priori estimate in an appropriate weighted $L^2$ norm, uniformly with respect to $\epsilon,\delta$; third, Proposition~\ref{prop:corrector} giving the details on the perturbed test functions. The proofs of those three auxiliary results are given in Section~\ref{sec:proofs}. Note that these three results are essential, and their proofs are given in a separate section since they are the most original technical contributions of this manuscript (compared with the more standard strategy described in Section~\ref{sec:proofdescription}).

\section{Setting}\label{sec:setting}

\subsection{Notation}

The solution $f^{\epsilon,\delta}$ of the stochastic kinetic problem considered in this article is a process $f^{\epsilon,\delta}:(t,x,v)\in[0,\infty)\times  \T^d\times V\to\R$, where $V$ is a measurable space, equipped with a $\sigma$-finite measure $\mu$, and $\T^d$ denotes the flat $d$-dimensional torus.

Let us introduce the standard Hilbert spaces $L^2_x=L^2(\T^d,\R)$ and $L^2=L^2(\T^d\times V,\R)$ of real-valued functions, with inner products defined as follows:
\[
	\lrpar{h,k}_{L^2_x} = \int_{\T^d} h(x) k(x) dx,\quad
	\lrpar{h,k}_{L^2} = \int_{\T^d \times V} h(x,v) k(x,v) dx d\mu(v).
\]
The associated norms are denoted by $\norm{\cdot}_{L_x^2}$ and $\norm{\cdot}_{L^2}$ respectively.

The following notation is used in the sequel: for all $f\in L^1(\T^d\times V)$, let $\rho\in L^1(\T^d)$ be defined by
\[
\rho\doteq \lrangle f \doteq \int_V f d\mu.
\]

In addition, for all $T \in (0,\infty)$, $\varsigma \in (0,1]$, $i \in \mathbb N_0$ and $p \in [1,\infty]$, introduce the Banach spaces $\C^i_x \doteq \C^i(\T^d,\R)$, $\C^0_T H^{-\varsigma}_x \doteq \C^0([0,T],H^{-\varsigma}(\T^d,\R))$ and $L^p_T L^2_x \doteq L^p([0,T],L^2(\T^d,\R))$.

Finally, the state space $E$ of the driving stochastic process is assumed to be a Banach space, with norm denoted by $\norm{\cdot}_E$.

\subsection{Assumptions on the coefficients}

Let us now state the assumptions concerning the linear operator $L$ and the mappings $a$, $b$ and $\sigma$.
\begin{assumption}\label{ass:coefficients}
    \begin{enumerate}
        \item Let the mapping $\M\in L^1(V,\mu)$, be such that $\M(v)>0$ for all $v\in V$, and normalized such that $\int_{V}\M(v)d\mu(v)=1$.
        \item The linear operator $L$ is defined as follows: for all $f\in L^1(V,\mu)$
        \begin{equation} \label{eq:def_BGK}
            Lf=\rho\M-f=\lrangle f \M-f.
        \end{equation}
        \item The mappings $a:V\to\R^d$ and $b:V\to\R^d$ are bounded. In addition $a$ satisfies the centering condition
        \[
        \int_V a(v)\M(v)d\mu(v)=0.
        \]
        \item The linear operators $A$ and $B$ are defined by
	\begin{gather*}
		Af(x,v) \doteq a(v) \cdot \nabla_x f(x,v),\\ 
		Bf(x,v) \doteq b(v) \cdot \nabla_x f(x,v). 
	\end{gather*}	
        \item The mapping $\sigma:E\to \C^{\floor{d/2}+2}_x$ is Lipschitz continuous.
    \end{enumerate}
\end{assumption}

Let us also introduce the weighted Hilbert space $\fSet$, with the inner product
\[
\lrpar{h,k}_{\fSet} = \int_{\T^d \times V} h(x,v) k(x,v) dx \frac{d\mu(v)}{\M(v)}.
\]
The associated norm is denoted by $\norm{\cdot}_{\fSet}$. Observe that applying the Cauchy-Schwarz inequality yields the following results: for all $f\in\fSet$, one has $f\in L^1(\T^d\times V)$, $\rho\in L_x^2$, and $Lf\in\fSet$, with
\begin{align*}
\iint_{\T^d\times V}|f(x,v)|d\mu(v)dx&\le \|f\|_{\fSet}\\
\|\rho\M\|_{\fSet}=\|\rho\|_{L_x^2}&\le \|f\|_{\fSet}.
\end{align*}

In addition, if $f_1\in L^2(\M^{-1})\cap L^2$ and $f_2\in L^2(\M)\cap L^2$, then
\[
\abs{\lrpar{f_1,f_2}_{L^2}} \leq \norm{f_1}_\fSet \norm{f_2}_{L^2(\M)}
\]
as a consequence of the Cauchy-Schwarz inequality.

\begin{example}
    The conditions in Assumption~\ref{ass:coefficients} above are satisfied in the following two examples:
    \begin{enumerate}
    \item Continuous velocities
        \begin{itemize}
            \item The space $V=\R^d$ is equipped with the Lebesgue measure $d\mu(v)=dv$.
            \item The function $\M(v)=(2\pi)^{-d/2}\exp(-\|v\|^2/2)$ for all $v\in\R^d$ is the standard Maxwellian.
            \item The velocity $a$ is a bounded odd function. For instance, relativistic particles satisfy $a(v) = \frac{v}{\sqrt{1 + \norm{v}^2}}$ in convenient units.

        \end{itemize}
    \item Discrete velocities
        \begin{itemize}
            \item The space $V=\set{\pm 1}^d$ is equipped with the counting measure.
            \item The function $\M$ is constant: $\M(v)=\frac{1}{2d}$ for all $v \in \set{\pm 1}^d$
            \item The velocity $a$ is an odd function. For instance, the isotropic discrete velocity is given by $a(v) = v$ for all $v \in \set{\pm 1}^d$.
        \end{itemize}
    \end{enumerate}
In both examples, the function $\sigma$ is defined either as
            \begin{itemize}
                \item $\sigma(\n)(x)=\sigma_1(\n(x))$ for all $\n\in E=\C^{\floor{d/2}+2}_x$ and $x\in \T^d$, with a mapping $\sigma_1:\R\to\R$,
                \item or as $\sigma(\n)(x)=\sigma_2(x,\n)$ for $\n\in E=\R$ and $x\in\T^d$, with a mapping $\sigma_2:\T^d\times\R \to \R$,
            \end{itemize}
            where $\sigma_1,\sigma_2$ are of class $\C^{\floor{d/2}+3}$ with bounded derivatives of all orders.    
    
\end{example}

\subsection{Assumptions on the driving process}

The driving process is a Markov process, which satisfies the conditions below.
\begin{assumption} \label{ass:m}
The family $\seq{m_\n(t)}{\n \in E}$ defines a $E$-valued Markov process, where one has the initial condition $m_\n(0) = \n$. Let $\L_m$ denote its infinitesimal generator, with domain denoted by $D(\L_m)$.

We assume that this Markov process is ergodic, and that its unique invariant distribution, denoted by $\nu$, is integrable: $\int_E \norm{\n} d\nu(\n) < \infty$.

The following notation is used throughout the article: for all Lipschitz continuous mappings $\theta:E\to\R$, set
\[
\overline{\theta}=\int_E\theta(\n)d\nu(\n).
\]
\end{assumption}

Let $\n_0 \in E$ be a given initial condition, in the sequel the value of $\n_0$ is omitted to simplify notation: for all $t\ge 0$, let
\[
m(t)=m_{\n_0}(t).
\]

Assumption~\ref{ass:m} is sufficient to state the main convergence results below. However, the analysis of the asymptotic behavior of $f^{\epsilon,\delta}$ when $\delta\to 0$ requires additional technical assumptions. Since they are not needed to state the convergence results below, they may be skipped by the reader, until they are used to prove auxiliary results.

\begin{assumption}\label{ass:moments_m_gamma}
The Markov process introduced in Assumption~\ref{ass:m} satisfies the appropriate moment bounds: there exists $\gamma \in (2,\infty)$ such that
	\begin{equation} \label{eq:moments_m_gamma}
		\sup_{i \in \mathbb N_0} \esp{\sup_{t \in [i,i+1]} \norm{m(t)}_E^\gamma} < \infty.
	\end{equation}
\end{assumption}
The assumption that $\gamma>2$ is crucial in the analysis. Observe that Assumption~\ref{ass:moments_m_gamma} implies the following results:
\begin{align*}
\int \norm{\n}^\gamma d\nu(\n)&<\infty\\
\underset{t\ge 0}\sup~\esp{\norm{m(t)}_E^\gamma}&<\infty.
\end{align*}

A mixing assumption is employed below to have quantitative information on the large time behavior of the driving process.
\begin{assumption}\label{ass:mixing}
The Markov process introduced in Assumption~\ref{ass:m} satisfies a mixing property: there exists a nonnegative function $\gamma_{\rm mix} \in L^1(\R^+)$ such that, for all initial conditions $\n_1, \n_2 \in E$, there exists a coupling $(m_{\n_1}^*,m_{\n_2}^*)$ of $(m_{\n_1},m_{\n_2})$, satisfying the inequality
	\begin{equation*}
		\esp{\norm{m_{\n_1}^*(t) - m_{\n_2}^*(t)}_E} \leq \gamma_{\rm mix}(t) \norm{\n_1-\n_2}_E
	\end{equation*}
for all $t\ge 0$.
\end{assumption}
Let us recall that a $E\times E$-valued random process $(m_{\n_1}^*,m_{\n_2}^*)$ is a coupling of $(m_{\n_1},m_{\n_2})$ if the marginals safisty $m_{\n_1}^* \stackrel{d}{=} m_{\n_1}$ and $m_{\n_2}^* \stackrel{d}{=} m_{\n_2}$ (where equality is understood in distribution). As a consequence of Assumption~\ref{ass:mixing}, it is straightforward to obtain the following result: if $\theta:E\to\R$ is Lipschitz continuous, then for all $\n\in E$ and all $t\ge 0$, one has
\[
\big|\E[\theta(m_\n(t))]-\overline{\theta}\big|\leq \max(1,\int_E \norm{\n'}_E d\nu(\n')){\rm Lip}(\theta)\gamma_{\rm mix}(t)(1+\norm{\n}_E),
\]
where ${\rm Lip}(\theta)$ denotes the Lipschitz constant of $\theta$. Due to this consequence of the mixing property (Assumption~\ref{ass:mixing}), the resolvent operator $R_0$ introduced below is well-defined.
\begin{definition}[Resolvent operator] \label{def:resolvent}
	Let $\lrpar{\C^{\floor{d/2}+2}_x}^*$ be the set of continuous linear forms on $\C^{\floor{d/2}+2}_x$ and let $E^*\!\lrpar{\sigma} \doteq \set{u \circ \sigma \mid u \in \lrpar{\C^{\floor{d/2}+2}_x}^*}\subset \Lip(E,\R)$. The resolvent operator $R_0$ is defined as follows: for all $\theta \in E^*\!\lrpar{\sigma}$
	\begin{equation*}
		R_0 (\theta - \bar \theta)(\n) = \int_0^\infty \esp{\theta(m_\n(t)) - \bar \theta} dt.
	\end{equation*}
	
	The function $\psi_{\theta}\doteq R_0(\theta-\bar{\theta})$ is the unique solution of the Poisson equation
	\begin{equation}
	-\L_m\psi_\theta=\theta-\bar \theta,
	\end{equation}
	satisfying the condition $\bar \psi_\theta=0$.
\end{definition}

Note that the functions $\psi_{\theta}$ satisfy the following bound: there exists $C\in(0,\infty)$, such that for all $\theta\in E^*\!\lrpar{\sigma}$ and for all $\n\in E$, one has
\begin{equation}\label{eq:estimate_resolvent}
|\psi_\theta(\n)|\leq C{\rm Lip}(\theta)(1+\norm{\n}_E),
\end{equation}

The remaining assumption deals with the infinitesimal generator $\L_m$ of the driving process.
\begin{assumption} \label{ass:regular_Lm}
	For all $\theta_1,\theta_2 \in E^*\!\lrpar{\sigma}$, assume that $\psi_{\theta_1} \psi_{\theta_2}$ is in the domain $D(\L_m)$ of the generator $\L_m$ of the driving process, and that $\L_m(\psi_{\theta_1}\psi_{\theta_2})$ has at most polynomial growth:
	\begin{equation*}
		\sup_{\n \in E} \frac{\abs{\L_m \lrpar{\psi_{\theta_1} \psi_{\theta_2}}(\n)}}{1 + \norm{\n}_E^2} < \infty.
	\end{equation*}
\end{assumption}

The assumptions above are satisfied if the driving process is a $E$-valued Ornstein-Uhlenbeck process, as explained below.
\begin{example}
	Let $\seq{m_\n(t)}{\n \in E}$ be defined by
	\begin{equation} \label{eq:def_OU}
		dm_{\n,t} = - (m_{\n,t} - \bar m) dt + dW_t, \quad m_{\n,0} = \n,
	\end{equation}
	where $W$ is an $E$-valued Wiener process. It satisfies Assumption \ref{ass:m} since it is ergodic and its unique invariant distribution is a normal distribution, hence integrable. Moreover, we have
	\begin{equation} \label{eq:expr_OU}
		m_\n(t) = \n e^{-t} + \bar m \lrpar{1 - e^{-t}} + \int_0^t e^{s-t} dW_s.
	\end{equation}
	Assumption \ref{ass:moments_m_gamma} is satisfied for any $\gamma \in (2,\infty)$. The coupling $(m_{\n_1}^*,m_{\n_2}^*)$ of Assumption \ref{ass:mixing} is obtained by driving both processes by the same Wiener process $W$. Indeed, \eqref{eq:expr_OU} becomes
	\begin{equation*}
		m_{\n_1}^*(t) - m_{\n_2}^*(t) = (\n_1 - \n_2) e^{-t},
	\end{equation*}
	and Assumption \ref{ass:mixing} is satisfied with $\gamma_{\rm mix}(t) = e^{-t}$. Finally, with the notation of Assumption \ref{ass:regular_Lm}, we have
	\begin{equation*}
		\psi_\theta(\n) = \int_0^\infty \lrpar{\theta(\n) - \theta(\bar m)} e^{-t} dt = \theta(\bar m) - \theta(\n).
	\end{equation*}
	Since the infinitesimal generator is given by $\L_m \phi(\n) = D\phi(\n) \cdot (\bar m - \n) + \frac{1}{2} \Tr\lrpar{D^2\phi(\n)}$, Assumption \ref{ass:regular_Lm} is also satisfied.
\end{example}

\section{Main result}\label{sec:results}

\subsection{Description of the model and of the limit problem}

The multiscale stochastic problem considered in this article depends on two parameters $\epsilon$ and $\delta$. Since the objective of this work is to prove a convergence result when $(\epsilon,\delta) \to (0,0)$, without loss of generality it is assumed that $\epsilon\in(0,\epsilon_0]$ and $\delta\in(0,\delta_0]$, where $\epsilon_0,\delta_0$ are fixed -- a precise condition is stated below. To simplify notation, we use the following convention: $\seq{X^{\epsilon,\delta}}{\epsilon,\delta}$ stands for the family of random variables $\seq{X^{\epsilon,\delta}}{\epsilon\in(0,\epsilon_0],\delta\in(0,\delta_0]}$.

First, for all $\delta\in(0,\delta_0]$, the fast driving process $m^\delta$ is defined as follows: for all $t\ge 0$, set
\begin{equation}
m^\delta(t)=m(t/\delta^2),
\end{equation}
where $m$ is the driving process given by Assumption~\ref{ass:m}, with the initial condition $m^\delta(0)=m(0)=\n_0$, which is assumed to be independent of $\delta$.

We study the asymptotic behavior when $(\epsilon,\delta) \to (0,0)$ of the solution $f^{\epsilon,\delta}$ of the following stochastic kinetic problem
\begin{equation} \label{eq:SPDE}
	\partial_t f^{\epsilon,\delta} + \Bigl(\frac{1}{\epsilon} a(v)+ b(v)\Bigr) \cdot \nabla_x f^{\epsilon,\delta} + \sigma(m^\delta) f^{\epsilon,\delta} = \frac{1}{\epsilon^2} Lf^{\epsilon,\delta}.
\end{equation}
with initial condition $f^{\epsilon,\delta}(0) = f^{\epsilon,\delta}_0$.

For any fixed $\epsilon\in(0,\epsilon_0]$, $\delta\in(0,\delta_0]$, the problem~\eqref{eq:SPDE} is globally well-posed in the following sense.
\begin{proposition} \label{prop:well-posed}
Introduce the linear operator $A^\epsilon=\epsilon^{-1}A+B$, with domain $D(A^\epsilon) \doteq \set{f \in \fSet \mid (x,v) \mapsto \bigl(\epsilon^{-1}a(v)+b(v)\bigr) \cdot \nabla_x f(x,v) \in \fSet}$, for all $\epsilon\in(0,\epsilon_0]$.

	Let $T \in (0,\infty)$, $\epsilon \in (0,\epsilon_0]$ and $\delta \in (0,\delta_0]$. For any $f^{\epsilon,\delta}_0 \in \fSet$, there exists, almost surely, a unique mild solution $f^{\epsilon,\delta}$ of~\eqref{eq:SPDE} in $\C^0([0,T];\fSet)$, in the sense that, almost surely, for all $t\in[0,T]$, one has
		\begin{equation*}
f^{\epsilon,\delta}(t) = e^{- t A^\epsilon} f^{\epsilon,\delta}_0 + \int_0^t e^{- (t-s) A^\epsilon} \lrpar{\frac{1}{\epsilon^2} Lf^{\epsilon,\delta}(s) - \sigma(m^\delta(s)) f^{\epsilon,\delta}(s)}ds.
	\end{equation*}
\end{proposition}

The proof of Proposition~\ref{prop:well-posed} is based on a standard fixed point argument, combined with the following observation:
\[
\underset{t\in[0,T]}\sup~\norm{m^{\delta}(t)}_E\le \underset{i\le T\delta^{-2}+1}\sup~\underset{t\in[i,i+1]}\sup~\norm{m(t)}_E<\infty
\]
owing to Assumption~\ref{ass:moments_m_gamma} on the moments of the driving process $m$. The proof of Proposition~\ref{prop:well-posed} is thus omitted.

Note that the statement of Proposition~\ref{prop:well-posed} is given for fixed $\epsilon>0$ and $\delta>0$, and does not provide uniform estimates of the solution $f^{\epsilon,\delta}$ with respect to these parameters. Proving such estimates needs extra arguments, which are not needed to state the main results of the article. We refer to Proposition... below for the statement of an appropriate a priori estimate in the $\fSet$ norm, which requires the introduction of a stopping time $\tau^\delta$ defined below.

Let us now introduce the so-called averaged equation. First, set
\begin{equation}\label{eq:def-KJsigma}
\begin{aligned}
K&=\int_V a(v)\otimes a(v)\M(v)d\mu(v)\in {\rm Sym}^+(d)\\
J&=\int_V b(v)\M(v)d\mu(v)\in\R^d\\
\bar\sigma &=\int_E \sigma(\n)d\nu(\n)\in \C^{\floor{d/2}+2}_x.
\end{aligned}
\end{equation}
Note that $K$ and $J$ are well-defined since $a$ and $b$ are assumed to be bounded (see Assumption~\ref{ass:coefficients}). In addition, $\bar{\sigma}$ is well-defined since $\sigma$ is globally Lipschitz continuous from $E$ to $\C^{\floor{d/2}+2}_x$ (see Assumption~\ref{ass:coefficients}) and since the probability distribution $\nu$ is integrable (see Assumption~\ref{ass:m}).

The unknown $\bar\rho$ of the averaged equation is a mapping defined on $[0,\infty)\times \T^d$. We are finally in position to write the averaged equation:
\begin{equation} \label{eq:limit_equation}
	\partial_t \bar\rho + J\cdot \nabla \bar\rho +\bar \sigma \, \bar\rho = \div (K\nabla \bar\rho),
\end{equation}
with initial condition $\bar\rho(0)=\bar\rho_0$. In~\eqref{eq:limit_equation} above, $\nabla$ and $\div$ are the gradient and divergence operators with respect to the variable $x$ respectively. We consider solutions of~\eqref{eq:limit_equation} in the weak sense, see Definition~\ref{def:limit_equation} below. Note that the solution may be a random process, even if the evolution is deterministic: it may happen that the initial condition $\bar\rho_0$ is random.
\begin{definition} \label{def:limit_equation}
	Let $T\in(0,\infty)$ and $\bar\rho_0$ be a $L^2_x$-valued random variable. A stochastic process $\bar\rho$ is a weak solution of~\eqref{eq:limit_equation} in $L^2_x$ if $\bar\rho \in L^\infty_T L^2_x$ almost surely and if, for all $\xi \in H^2_x$ and $t \in [0,T]$, almost surely,
	\begin{equation} \label{eq:limit_equation_weak}
		\lrpar{\bar\rho(t),\xi}_{L^2_x} = \lrpar{\bar\rho_0,\xi}_{L^2_x} + \int_0^t \lrpar{\bar\rho(s),\div(K \nabla \xi) + J \cdot \nabla \xi - \bar \sigma \xi}_{L^2_x} ds.
	\end{equation}
\end{definition}
For any $L_x^2$-valued random variable, the averaged equation~\eqref{eq:limit_equation} admits a unique weak solution in the sense of Definition~\ref{def:limit_equation}.

\subsection{Convergence results}

Let us now state the main results of this article, concerning the asymptotic behavior of $\rho^{\epsilon,\delta}$ and $f^{\epsilon,\delta}$ when $(\epsilon,\delta) \to (0,0)$. In the sequel, it is convenient to impose the following non restrictive conditions on the parameters $\epsilon_0$ and $\delta_0$ (such that $\epsilon\in(0,\epsilon_0]$ and $\delta\in(0,\delta_0]$):
\begin{equation} \label{eq:def_epsilon_0}
		\epsilon_0 < \min\lrpar{1,\lrpar{4 (\norm{a}_{L^\infty} + \norm{b}_{L^\infty}) (1 + T \norm{\bar \sigma}_{\C^1_x})}^{-1}},\quad \delta_0 < \min\lrpar{1,\norm{\n_0}_E^{-1}}.
\end{equation}
Note that the condition for $\delta_0$ depends on the initial condition $\n_0$ of the driving process. This is one of the reasons why it is convenient to assume that it is deterministic and that it does not depend on $\delta$. Extensions to more general initial conditions for the driving process would require extra technical assumptions and computations, which are omitted to simplify the setting and focus on the main aspects of the analysis of the asymptotic behavior of the stochastic multiscale problem~\eqref{eq:SPDE} when $(\epsilon,\delta) \to (0,0)$.

The initial condition $f_0^{\epsilon,\delta}$ of~\eqref{eq:SPDE} is assumed to satisfy the following conditions.

\begin{assumption} \label{ass:f0}
	The family of initial conditions $\seq{f^{\epsilon,\delta}_0}{\epsilon,\delta}$ satisfies the following moment bound:
	\begin{equation*}
		\sup_{\epsilon\in(0,\epsilon_0],\delta\in(0,\delta_0]} \esp{\norm{f^{\epsilon,\delta}_0}_\fSet^{12}} < \infty.
	\end{equation*}
	In addition, the initial density $\rho^{\epsilon,\delta}_0 \in L^2_x$ converges in distribution when $(\epsilon,\delta) \to (0,0)$ in $L^2_x$ to $\bar\rho_0 \in L^2_x$: recall that this means that for any bounded continuous mapping $\Phi:L^2_x\to\R$, one has
	\[
	\esp{\Phi(\rho^{\epsilon,\delta}_0)}\xrightarrow[\epsilon,\delta \to 0]{} \esp{\Phi(\bar\rho_0)}.
	\]
\end{assumption}
Observe that in general $\bar\rho_0$ is a $L_x^2$-valued random variable.

We are now in position to state the two main convergence results of this article. First, see Theorem~\ref{thm:main_Hneg}, $\rho^{\epsilon,\delta}=\lrangle {f^{\epsilon,\delta}}$ converges in distribution to the unique solution $\bar\rho$ of the averaged equation~\eqref{eq:limit_equation}, in the space $\C^0_T H^{-\varsigma}_x$ for all $\varsigma>0$. Second, under an additional technical assumption which allows to apply an averaging lemma, one obtains a stronger result, see Theorem~\ref{thm:main_L2}: $\rho^{\epsilon,\delta}$ converges in distribution to $\bar\rho$ in the space $L^2_T L^2_x$, and $f^{\epsilon,\delta}$ converges in distribution to $\bar\rho\M$ in the space $L_T^2\fSet$. Moreover, the convergence results hold in probability if the initial condition $\bar\rho_0$ (given by Assumption~\ref{ass:f0}) of the averaged equation~\eqref{eq:limit_equation} is deterministic.

Let us state the first main result of this article.
\begin{theorem} \label{thm:main_Hneg}
	Let Assumptions \ref{ass:coefficients} to \ref{ass:f0} be satisfied. Let $f^{\epsilon,\delta}$ be the solution of~\eqref{eq:SPDE}. Then, when $(\epsilon,\delta) \to (0,0)$, the random variable $\rho^{\epsilon,\delta}=\lrangle{f^{\epsilon,\delta}}$ converges in distribution to the unique weak solution $\bar\rho$ of~\eqref{eq:limit_equation}, in $\C^0_T H^{-\varsigma}_x$.
	
	In addition, if the initial condition $\bar\rho_0=\underset{(\epsilon,\delta) \to (0,0)}\lim~\rho_0^{\epsilon,\delta}\in L^2_x$ is deterministic, then the convergence of $\rho^{\epsilon,\delta}$ to $\bar\rho$ holds in probability.
\end{theorem}

An additional assumption is required to state Theorem~\ref{thm:main_L2}, as in~\cite{rakotonirina2020diffusion} (Assumption~\ref{ass:avlemma} below). It allows us to apply a so-called averaging lemma (see \cite[Theorem 2.3]{bouchut1999averaging}), developed for the study of kinetic PDEs, and to obtain convergence in the space $L_x^2$ (in Theorem~\ref{thm:main_L2}) instead of $H^{-\varsigma}$ (in Theorem~\ref{thm:main_Hneg}).
\begin{assumption}\label{ass:avlemma}
	\begin{itemize}
	\item The space for the velocity variable is given by $(V,d\mu) = (\R^n,\frac{d\mu}{dv}(v)dv)$,  with a Radon-Nikodym derivative (with respect to Lebesgue measure) satisfying $\frac{d\mu}{dv} \in H^1(\R^n)$.
	\item The mapping $a$ is locally Lipschitz continuous
	\item There exist $C_{a,\mu} \in (0,\infty)$ and $\varsigma_{a,\mu} \in (0,1]$ such that, for all $u \in S^{d-1}$, all $\lambda\in \R$ and all $\eta\in(0,\infty)$, one has
	\begin{equation*}
		 \int_{\lambda < a(v) \cdot u < \lambda + \eta} \lrpar{\abs{\frac{d\mu}{dv}(v)}^2 + \abs{\nabla \frac{d\mu}{dv}(v)}^2} dv \leq C_{a,\mu} \eta^{\varsigma_{a,\mu}}.
	\end{equation*}
	\end{itemize}
\end{assumption}

We are now in position to state the second main result of this article.
\begin{theorem} \label{thm:main_L2}
	Let Assumptions \ref{ass:coefficients} to \ref{ass:f0} and Assumption~\ref{ass:avlemma} be satisfied.
	
	Then, when $(\epsilon,\delta) \to (0,0)$, the random variable $\rho^{\epsilon,\delta}=\lrangle{f^{\epsilon,\delta}}$ converges in distribution to the unique weak solution $\bar\rho$ of~\eqref{eq:limit_equation}, in the space $L_T^2L_x^2$. Moreover, when $(\epsilon,\delta) \to (0,0)$, the random variable $f^{\epsilon,\delta}$ converges in distribution to $\bar\rho \M$, in the space $L^2_T \fSet$.
	
	In addition, if the initial condition $\bar\rho_0=\underset{(\epsilon,\delta) \to (0,0)}\lim~\rho_0^{\epsilon,\delta}\in L^2_x$ is deterministic, then the convergence of $\rho^{\epsilon,\delta}$ to $\bar\rho$ and of $f^{\epsilon,\delta}$ to $\bar\rho\M$ hold in probability.
\end{theorem}

\subsection{Discussion}

The main results, Theorems~\ref{thm:main_Hneg} and~\ref{thm:main_L2}, state that diffusion approximation (in the PDE sense) and averaging (in the probabilistic sense) results hold, when $(\epsilon,\delta)\to (0,0)$. These results are natural generalizations of previously obtained results, either in the deterministic case ($\epsilon\to 0$, $\sigma=0$), or in the probabilistic case ($\delta\to 0$, $\epsilon>0$ fixed). In fact, using the same arguments as in Section~\ref{sec:proofdescription} below, one may obtain the following results, where the limits $\epsilon\to 0$ and $\delta\to 0$ are taken successively.

On the one hand, if $\delta>0$ is fixed, then $\rho^{\epsilon,\delta}=\lrangle f^{\epsilon,\delta}$ converges when $\epsilon\to 0$, to the solution $\rho^{0,\delta}$ of the evolution equation
\begin{equation*}
\partial_t \rho^{0,\delta} + \sigma(m^\delta) \rho^{0,\delta} = \div(K \nabla \rho^{0,\delta}) - J \cdot \nabla \rho^{0,\delta}.
\end{equation*}
That result is a standard diffusion approximation result in the PDE sense. Then, when $\delta\to 0$, $\rho^{0,\delta}$ converges to the solution $\bar\rho$ of the limit equation~\eqref{eq:limit_equation}, owing to the standard averaging principle for stochastic problems.

On the other hand, if $\epsilon>0$ is fixed, then owing to the standard averaging principle, $f^{\epsilon,\delta}$ converges when $\delta\to 0$, to the solution $f^{\epsilon,0}$ of the evolution equation
\begin{equation*}
	\partial_t f^{\epsilon,0} + \frac{1}{\epsilon} a(v) \cdot \nabla_x f^{\epsilon,0} + b(v) \cdot \nabla_x f^{\epsilon,0} + \bar \sigma f^{\epsilon,0} = \frac{1}{\epsilon^2} Lf^{\epsilon,0}.
\end{equation*}
Then, when $\epsilon\to 0$, $\rho^{\epsilon,0}=\lrangle f^{\epsilon,0}$ converges to the solution $\bar\rho$ of the limit equation~\eqref{eq:limit_equation}, owing to the standard diffusion approximation result in the PDE sense.

To the best of our knowledge, the results above have not been rigorously proved in the literature, however they are variants of well-studied results. The proofs of Theorems~\ref{thm:main_Hneg} and~\ref{thm:main_L2} do not encompass those regimes when either $\epsilon=0$ or $\delta=0$: essentially this would require to adapt the construction of the perturbed test function. Indeed, below we directly focus on the behavior of $\rho^{\epsilon,\delta}$, thus the convergence $f^{\epsilon,\delta}\to f^{\epsilon,0}$ when $\delta\to 0$ cannot be covered directly, for instance. In addition, one of the arguments of the proofs is the convergence $\tau^\delta\to\infty$ when $\delta\to 0$ (where $\tau^\delta$ are stopping times defined below), thus the convergence $\rho^{\epsilon,\delta}\to\rho^{0,\delta}$ when $\epsilon\to 0$ cannot be covered without substantial modifications.

Still, one obtains the following result:
\[
\underset{\delta\to 0}\lim~\underset{\epsilon\to 0}\lim~\rho^{\epsilon,\delta}=\underset{\epsilon\to 0}\lim~\underset{\delta\to 0}\lim~\rho^{\epsilon,\delta}=\underset{(\epsilon,\delta)\to (0,0)}\lim\rho^{\epsilon,\delta}=\bar\rho,
\]
where the convergence is understood in the appropriate sense. The analysis presented in this article thus departs from the setting of~\cite{debussche2011diffusion,debussche2017diffusion,debussche2020diffusion,rakotonirina2020diffusion}.
 where in all cases there is only one small parameter $\epsilon=\delta$. Our result is expected but important: it shows that the diffusion approximation and the averaging principle can be decoupled. As already mentioned in the introduction, considering the general case $\epsilon\neq \delta$ requires new arguments, in particular the construction of the perturbed test functions needs an additional corrector.

In the setting, it is assumed that the initial condition $m^\delta(0)$ of the fast driving process is deterministic and independent of $\delta$: $m^\delta(0)=m_0$. It would be possible to extend the results to more general initial conditions, under appropriate modified moment conditions. This would for instance allow us to include the case where $m^\delta(0)$ is random and distributed following the ergodic invariant distribution $\nu$. Note also that if either the mapping $\sigma$ or the process $m$ would have a bounded support, the analysis would be simplified.

As explained in the introduction, we have left open several questions for future works: first, the analysis of fluctuations, second, the identification of (strong and weak) rates of convergence.

\section{Description of the proof}\label{sec:proofdescription}

In the sequel, the following convention is employed: given variables $u$ and parameters $\lambda$, the notation $X_1(u)\lesssim X_2(u)$ means that for all parameters $\lambda$, there exists $C(\lambda)\in(0,\infty)$ such that one has $X_1(u)\le C(\lambda)X_2(u)$ for all $u$. From the context the identification of variables (typically $\epsilon,\delta,f,\ell$) and parameters (typically $T,\varphi$) will be clear.

First, we introduce two of the most important tools of the proofs of the main results: the stopping time $\tau^\delta$ (see Section~\ref{sec:stoppingtime}), and the perturbed test function $\varphi^{\epsilon,\delta}=\varphi+\epsilon\varphi_{1,0}+\epsilon^2\varphi_{2,0}+\delta^2\varphi_{0,2}+\epsilon\delta^2\varphi_{1,2}$ (see Section~\ref{sec:ptf}), constructed for a class of admissible functions $\varphi$ such that $\varphi(f,\n)=\varphi(\lrangle f)=\varphi(\rho)$. Sections~\ref{sec:stoppingtime} and~\ref{sec:ptf} contain important auxiliary results, which require technical arguments in their proofs: the proofs are therefore postponed to Section~\ref{sec:proofs}.

The use of the stopping time $\tau^\delta$ is instrumental to obtain appropriate moment bounds of the solutions, uniformly with respect to $\epsilon,\delta$. We prove that $\tau^\delta\to \infty$ in probability when $\delta\to 0$. In the arguments, as a consequence of Slutsky's lemma (see more details below), it is then sufficient to consider the stopped processes defined by $f^{\epsilon,\delta,\tau^\delta}=f^{\epsilon,\delta}(\cdot\wedge\tau^\delta)$, $\rho^{\epsilon,\delta,\tau^\delta}=\rho^{\epsilon,\delta}(\cdot\wedge\tau^\delta)$ and $m^{\delta,\tau^\delta}=m^{\delta}(\cdot\wedge\tau^\delta)$.

The use of the perturbed test function method is standard in the analysis of multiscale stochastic problems, see for instance~ \cite{debussche2011diffusion,debussche2017diffusion,debussche2020diffusion,rakotonirina2020diffusion} in a similar context of stochastic kinetic equations. Compared to those references, note that it is necessary to consider two parameters $\epsilon,\delta$ which may be independent. When $\delta=\epsilon$, the construction of the corrector $\varphi_{1,2}$ is not needed.

We then proceed to the proof of the main results of this article. The arguments are standard. We first check a tightness property for $\seq{\rho^{\epsilon,\delta}}{\epsilon,\delta,\tau^\delta}$ in the appropriate function space. We then check that the Markov process $(f^{\epsilon,\delta,\tau^\delta},m^{\delta,\tau^\delta})$ is the solution of a martingale problem for all $\epsilon>0,\delta>0$, and letting $(\epsilon,\delta) \to (0,0)$, using the perturbed test function, we prove that any limit point of the family $\seq{\rho^{\epsilon,\delta,\tau^\delta}}{\epsilon,\delta}$ is a weak solution of the averaged equation, in the sense of~\eqref{eq:limit_equation_weak}. Since $\tau^\delta\to\infty$ in probability when $\delta\to 0$, a uniqueness argument for the averaged equation then concludes the proof of Theorem~\ref{thm:main_Hneg}. Theorem~\ref{thm:main_L2} is then obtained by the application of the averaging lemma. The fact that convergence holds in probability when $\bar\rho_0$ is deterministic is a straightforward well-known consequence of Portmanteau Theorem (and is not specific to the PDE framework of this paper): in that case the solution $\bar\rho$ of the averaged equation is also deterministic.

\subsection{Stopping times and a priori estimates}\label{sec:stoppingtime}

In this section, first we give the definition of the stopping time $\tau^\delta$, second we state that $\tau^\delta\to \infty$ in probability when $\delta\to 0$. Finally, we state an a priori estimate for $f^{\epsilon,\delta,\tau^\delta}$ in $\fSet$, which is uniform with respect to $\epsilon,\delta$. The proofs are postponed to Section~\ref{sec:proofs}.

For all $\delta\in(0,\delta_0]$ and all $t\ge 0$, set
\begin{equation} \label{eq:def_zeta_epsilon}
		\zeta^\delta(t) = \frac{1}{\delta} \int_0^t \lrpar{\sigma(m^\delta(s)) - \bar \sigma} ds \in \C^{\floor{d/2}+2}_x.
	\end{equation}

\begin{definition}\label{def:taue}
Let $\alpha \in (\frac{2}{\gamma},1)$, then for all $\delta\in(0,\delta_0]$, define
\begin{equation*} 
		\tau^\delta = \tau^\delta_m \wedge \tau^\delta_\zeta,
	\end{equation*}
	where
	\begin{gather*}
		\tau^\delta_m = \inf \set{t \in [0,T] \mid \norm{m^\delta(t)}_E \geq \delta^{-\alpha}},\\
		\tau^\delta_\zeta = \inf \set{t \in [0,T] \mid \norm{\zeta^\delta(t)}_{\C^1_x} \geq \delta^{-1}}. 
	\end{gather*}
\end{definition}
Note that $\tau^\delta$ is a stopping time for the filtration $\seq{\mathcal F^\delta_t}{t \in \R^+}$ generated by the driving process $m^\delta$, with $\mathcal F_t^\delta=\sigma\seq{m^\delta(s)}{0\le s\le t}$. In the definition above, $\gamma$ is the parameter introduced in Assumption~\ref{ass:moments_m_gamma}.

The initial conditions of the stopped processes $m^{\delta,\tau^\delta}$ and $\zeta^{\delta,\tau^\delta}$ satisfy $\norm{m\dstopped(0)}_E=\norm{\n_0}_E\le \delta_0^{-1}\le \delta^{-\alpha}$ (since $\delta\le \delta_0\le 1$ and $\alpha<1$), and $\zeta\dstopped(0)=0$. As a consequence, almost surely $\tau^\delta>0$, and the following estimates for $m^\delta$ and $\zeta^\delta$ hold: for all $t\ge 0$ and all $\delta\in(0,\delta_0]$, one has
\begin{equation} \label{eq:taue_bounds_m}
		\norm{m\dstopped(t)}_E = \norm{m^\delta(t \wedge \tau^\delta)}_E \leq \delta^{-\alpha},
\end{equation}
and
\begin{equation} \label{eq:taue_bounds_zeta}
		\norm{\zeta\dstopped(t)}_{\C^1_x} = \norm{\zeta^\delta(t \wedge \tau^\delta)}_{\C^1_x} \leq \delta^{-1}.
	\end{equation}

Let us now study the behavior of the stopping time $\tau^\delta$ when $\delta\to 0$.
\begin{proposition} \label{prop:taue_to_infty}
	When $\delta\to 0$, $\tau^\delta\to\infty$ in probability: for all $T>0$, one has
	\begin{equation*} 
		\proba{\tau^\delta < T} \xrightarrow[(\epsilon,\delta) \to 0]{} 0.
	\end{equation*}
\end{proposition}
The proof of Proposition~\ref{prop:taue_to_infty} is postponed to Section~\ref{sec:proofs}.

Finally, let us state the following a priori estimate in the $\fSet$ norm for $f\edstopped(t)$, in an almost sure sense.
\begin{proposition} \label{prop:L2_bound}
	For $T \in (0,\infty)$, there exists $C(T)\in(0,\infty)$, such that for all $t \in [0,T]$, $\epsilon \in (0,\epsilon_0]$ and $\delta \in (0,\delta_0]$, almost  surely one has
	\begin{equation} \label{eq:L2_bound}
		\norm{f\edstopped(t)}_\fSet^2 + \frac{1}{2\epsilon^2} \int_0^{t \wedge \tau^\delta} \norm{Lf\edstopped(s)}_\fSet^2 ds \leq C(T) \norm{f^{\epsilon,\delta}_0}_\fSet^2.
	\end{equation}
\end{proposition}

Let us emphasize that the constant $C(T)$ appearing on the right-hand side of~\eqref{eq:L2_bound} is deterministic and does not depend on $\epsilon$ and $\delta$.

The proof of Proposition~\ref{prop:L2_bound} is posponed to Section~\ref{sec:proofs}. Note that the a priori estimate $\norm{f\edstopped(t)}_\fSet\lesssim \norm{f^{\epsilon,\delta}_0}_\fSet$ is instrumental in all the arguments of the proof of Theorem~\ref{thm:main_Hneg}, whereas the upper bound for the integral term on the left-hand side of~\eqref{eq:L2_bound} is used only in the proof of Theorem~\ref{thm:main_L2}.

\subsection{Perturbed test functions}\label{sec:ptf}

In this section, we describe the construction of a perturbed test function $\varphi^{\epsilon,\delta}$, such that the following properties are satisfied:
\begin{align*}
\varphi^{\epsilon,\delta}&\xrightarrow[\epsilon,\delta \to 0]{} \varphi\\
\L^{\epsilon,\delta}\varphi^{\epsilon,\delta}&\xrightarrow[\epsilon,\delta \to 0]{} \L\varphi,
\end{align*}

where $\varphi$ is any sufficiently smooth function such that $\varphi(f,\n)=\varphi(\lrangle f)$, and $\L^{\epsilon,\delta}$ and $\L$ are the infinitesimal generators associated with the Markov processes $(f^{\epsilon,\delta},m^\delta)$ and $\bar\rho$, solving~\eqref{eq:SPDE} and~\eqref{eq:limit_equation} respectively.

To state more rigorously and more precisely the properties mentioned above, we first introduce two appropriate classes of test functions, such that $\varphi\in\Theta_{\lim}$ and $\varphi^{\epsilon,\delta}\in\Theta$ for all $\epsilon,\delta$, and such that the errors $|\varphi^{\epsilon,\delta}-\varphi|$ and $|\L^{\epsilon,\delta}\varphi^{\epsilon,\delta}-\L\varphi|$ are quantified in terms of $\epsilon,\delta$.

Let us first describe the class of functions $\Theta_{\lim}$.
\begin{definition} \label{def:Theta-lim}
	Let $\Theta_{\lim}$ be the class of real-valued test functions $\phi$ such that for all $f\in\fSet$ and $\n\in E$, one has
	\begin{equation}\label{eq:phi-Theta-lim}
		\phi(f,\n) = \phi(\rho) = \chi(\lrpar{\rho,\xi}_{L^2_x}),
	\end{equation}
where we recall that $\rho = \lrangle f = \int_V f d\mu$, with arbitrary $\chi \in \C^3_b(\R,\R)$ and $\xi \in \C^3_x$.
\end{definition}
Recall that $K$, $J$ and $\bar\sigma$ are defined by~\eqref{eq:def-KJsigma}. The infinitesimal generator $\L$ associated with the limit problem~\eqref{eq:limit_equation} is defined by
	\begin{equation} \label{eq:def_L}
		\L \phi(\rho) = D \phi (\rho) \cdot \lrpar{\div_x(K\nabla_x \rho) - J \cdot \nabla_x \rho - \bar \sigma \rho},
	\end{equation}
for all $\rho\in L_x^2$ and all $\varphi\in D(\L)$, with the domain $D(\L)$ given by
\[
D(\L) = \set{\phi \in \C^0(L^2_x) \mid \L \phi \in \C^0(L^2_x)}.
\]
Note that $\Theta_{\lim}\subset D(\L)$. In addition, for all $\varphi\in \Theta_{\lim}$, of the form~\eqref{eq:phi-Theta-lim}, one has for all $\rho\in L_x^2$
\begin{equation}\label{eq:Lrho-theta}
    \L\phi(\rho)= \chi'(\lrpar{\rho,\xi}_{L^2_x}) \lrpar{ \rho , \div_x(K\nabla_x \xi) + J \cdot \nabla_x \xi - \bar \sigma \xi }_{L^2_x}.
\end{equation}

Let us now describe the class of functions $\Theta$.
\begin{definition} \label{def:Theta}
	Let $\Theta$ be the class of test functions $\phi : \fSet \times E \to \R$ satisfying the following conditions.
	\begin{itemize}
		\item For all $\n\in E$, $\phi(\cdot,\n) \in \C^1(\fSet)$.
		\item For all $f \in \fSet$, $\phi(f,\cdot) \in \C^0(E)$.
		\item For all $i \in \set{1,2}$ and all $f \in \fSet$, $\phi(f,\cdot)^i \in D(\L_m)$ and $\L_m(\phi^i) \in \C^0(\fSet \times E)$.
		\item For all $f \in \fSet$ and $\n \in E$, denote by $\nabla_f \phi(f,\n) \in \fSet$ the gradient of $\phi$ at $(f,\n)$, which is defined defined such that
		\begin{equation*}
			\forall h \in \fSet, \lrpar{\nabla_f \phi(f,\n),h}_{\fSet} = D_f \phi(f,\n) \cdot h.
		\end{equation*}
		Then, for all $f \in \fSet$ and $\n\in E$, one has
		\begin{equation*} 
			\int\int \norm{\nabla_x \nabla_f \phi(f,\n)(x,v)}^2 dx \frac{d\mu(v)}{\M(v)} < \infty.
		\end{equation*}
		\item There exists $C_\phi (0,\infty)$ such that, for all $f, h \in \fSet$ and $\n_1, \n_2 \in E$,
		\begin{multline*} 
			\abs{\phi(f,\n_1)} + \abs{D_f \phi(f,\n_1)(Ah)} + \abs{D_f \phi(f,\n_1)(Bh)} + \abs{D_f \phi(f,\n_1)(\n_2 h)} + \abs{D_f \phi(f,\n_1)(Lh)}\\
			\leq C_\phi \lrpar{1 + \norm{f}_\fSet^3 + \norm{h}_\fSet^3} \lrpar{1 + \norm{\n_1}_E^2 + \norm{\n_2}_E^2}.
		\end{multline*}
	\end{itemize}
\end{definition}

The infinitesimal generator $\L^{\epsilon,\delta}$ associated with the stochastic problem~\eqref{eq:SPDE} with driving process $m^\delta$, satisfies the following multiscale expansion in terms of $\epsilon,\delta$:
\begin{equation}
	\L^{\epsilon,\delta} = \L_0 + \epsilon^{-1} \L_1 + \epsilon^{-2} \L_2 + \delta^{-2} \L_m
\end{equation}
where the infinitesimal generator of the driving process $\L_m$ is introduced in Assumption~\ref{ass:m} above, and $\L_0$, $\L_1$ and $\L_2$ are defined as follows: for all $f\in\fSet$ and all $\n\in E$, set
\begin{gather*}
	\L_0 \phi(f,\n) = - D_f \phi(f,\n) \cdot (\sigma(\n) f + Bf),\\
	\L_1 \phi(f,\n) = - D_f \phi(f,\n) \cdot (Af),\\
	\L_2 \phi(f,\n) = D_f \phi(f,\n) \cdot Lf,
\end{gather*}
for real-valued functions $\varphi\in D(\L^{\epsilon,\delta})$ where
\[
D(\L^{\epsilon,\delta})=\set{\phi \in \C^0(\fSet\times E) \mid \L^{\epsilon,\delta} \phi \in \C^0(\fSet\times E)}
\]
is the domain of the unbounded linear operator $\L^{\epsilon,\delta}$. Recall that $A=a(v)\cdot \nabla_x$ and $B=b(v)\cdot\nabla_x$ (see Assumption~\ref{ass:coefficients}).

Note that the following property is satisfied: for all $\varphi\in\Theta$, one has $\varphi\in D(\L^{\epsilon,\delta})$ and $\varphi^2\in D(\L^{\epsilon,\delta})$. In addition, observe that $\Theta_{\lim}\subset\Theta$.

We are now in position to state the main result of this section.
\begin{proposition} \label{prop:corrector}
For all $\phi \in \Theta_{\lim}$, there exists functions $\phi_{1,0}$, $\phi_{2,0}$, $\phi_{0,2}$ and $\phi_{1,2}$ such that the following properties hold.
\begin{itemize}
\item Set
	\begin{equation} \label{eq:def_phi_epsilon}
		\phi^{\epsilon,\delta} = \phi + \epsilon \phi_{1,0} + \epsilon^2 \phi_{2,0} + \delta^2 \phi_{0,2} + \epsilon \delta^2 \phi_{1,2}.
	\end{equation}
Then $\phi^{\epsilon,\delta} \in \Theta$.
\item There exists $C(\varphi)\in(0,\infty)$, such that for all $\epsilon\in(0,\epsilon_0]$, $\delta\in(0,\delta_0]$ and all $(f,\n)\in\fSet\times E$, one has
\begin{equation}
		\begin{aligned} \label{eq:control_perturbed_generator}
		\abs{\L^{\epsilon,\delta} \phi^{\epsilon,\delta} - \L \phi}(f,\n)
			&\leq C(\varphi)  (1 + \norm{f}_\fSet^3)\Bigl( \epsilon(1 + \norm{\n}_E)+\delta^2 (1 + \norm{\n}_E^2)\Bigr).
		\end{aligned}
\end{equation}
\item One has the following upper bounds: there exists $C(\varphi)\in(0,\infty)$, such that for all $(f,\n)\in\fSet\times E$,
	\begin{gather}
		\label{eq:control_phi_10}
		\abs{\phi_{1,0}(f)} \leq C(\varphi) \norm{f}_\fSet,\\
		\label{eq:control_phi_20}
		\abs{\phi_{2,0} (f)} \leq C(\varphi) (1 + \norm{f}_\fSet^2),\\
		\label{eq:control_phi_02}
		\abs{\phi_{0,2} (f,\n)} \leq C(\varphi) \norm{f}_\fSet (1 + \norm{\n}_E),\\
		\label{eq:control_phi_12}
		\abs{\phi_{1,2} (f,\n)} \leq C(\varphi) (1 + \norm{f}_\fSet^2) (1 + \norm{\n}_E).
		\end{gather}
\end{itemize}
\end{proposition}

Note that the following error estimates holds: for all $(f,\n)\in \fSet\times E$, one has
\begin{equation}\label{eq:control_error_phi}
\abs{\varphi^{\epsilon,\delta}(f,\n)-\varphi(\lrangle f)}\lesssim_\varphi (1 + \norm{f}_\fSet^2) (\epsilon + \delta^2 (1 + \norm{\n}_E)).
\end{equation}

The proof of Proposition~\ref{prop:corrector} is postponed to Section~\ref{sec:proofs}. Using the standard terminology, the functions $\varphi_{1,0}$, $\varphi_{2,0}$, $\varphi_{0,2}$ and $\varphi_{1,2}$ are referred to as the correctors in the sequel.

Note that one of the novelties of the result above is the construction of the corrector $\varphi_{1,2}$, which is not required in the case $\epsilon=\delta$ which is treated in other contributions, see~\cite{debussche2011diffusion,debussche2017diffusion,debussche2020diffusion,rakotonirina2020diffusion}. More precisely, if $\epsilon,\delta$, it suffices to construct a perturbed test function of the type $\varphi^{\epsilon}=\varphi+\epsilon\varphi_1+\epsilon^2\varphi_2$. In fact, $\varphi_1=\varphi_{1,0}$ and $\varphi_2=\varphi_{2,0}+\varphi_{0,2}$.

\subsection{Martingale property}

The proofs of Theorems~\ref{thm:main_Hneg} and~\ref{thm:main_L2} is based on an interpretation in terms of solutions of martingale problems. Indeed, combined with the perturbed test function approach described above, this formulation is convenient to identify limit points when $(\epsilon,\delta) \to (0,0)$.

Several arguments in the proofs of the auxiliary results below employ the following auxiliary result.

\begin{proposition} \label{prop:martingale_formulation_epsilon}
	Let $\phi \in \Theta$. For all $t \in \R^+$, set
	\begin{equation} \label{eq:martingale_formulation_epsilon}
		M_\phi^{\epsilon,\delta}(t) \doteq \phi(f^{\epsilon,\delta}(t),m^\delta(t)) - \phi(f^{\epsilon,\delta}(0),m^\delta(0)) - \int_0^t \L^{\epsilon,\delta} \phi (f^{\epsilon,\delta}(s),m^\delta(s)) ds.
	\end{equation}

	Then, $M\edstopped_\phi$ is a càdlàg $\seq{\mathcal F^\delta_t}{t \in \R^+}$-martingale. In addition, for all $t \in \R^+$, one has
	\begin{align*}
		\esp{\abs{M\edstopped_\phi(t)}^2}
			&= \esp{\int_0^{t \wedge \tau^\delta} \lrpar{\L^{\epsilon,\delta} (\phi^2) - 2 \phi \L^{\epsilon,\delta} \phi}(f^{\epsilon,\delta}(s),m^\delta(s))ds}\\
			&= \frac{1}{\delta^2} \esp{\int_0^{t \wedge \tau^\delta} \lrpar{\L_m (\phi^2) - 2 \phi \L_m \phi}(f^{\epsilon,\delta}(s),m^\delta(s))ds}.
	\end{align*}
\end{proposition}
In the statement of Proposition~\ref{prop:martingale_formulation_epsilon} above, note that the process $M^{\epsilon,\delta}_\phi$ is stopped, where the stopping time $\tau^{\delta}$ is given by Definition~\ref{def:taue}. Considering the stopped process allows us to use the estimate~\eqref{eq:L2_bound} of Proposition~\ref{prop:L2_bound} in the sequel. Moreover, note that $\phi$ is assumed to belong to the class of functions $\Theta$ introduced in Definition~\ref{def:Theta}: in fact Proposition~\ref{prop:martingale_formulation_epsilon} is the justification of the requirements on $\phi$ in Definition~\ref{def:Theta}.

The proof of Proposition~\ref{prop:martingale_formulation_epsilon} is standard and is omitted.

\subsection{Proofs of Theorems~\ref{thm:main_Hneg} and~\ref{thm:main_L2}}

In order to prove the convergence in distribution results of $\rho^{\epsilon,\delta}$ to $\bar\rho$ when $(\epsilon,\delta)\to(0,0)$, it suffices to prove that for any arbitrary sequence $\seq{\epsilon_i,\delta_i}{i\ge 1}$ such that $(\epsilon_i,\delta_i)\to (0,0)$ when $i\to \infty$, $\rho^{\epsilon_i,\delta_i}$ converges in distribution to $\bar\rho$. To simplify the notation, we fix such a sequence and in the sequel one should interpret $\epsilon=\epsilon_i$ and $\delta=\delta_i$. Moreover, in Proposition~\ref{prop:limit_martingale} and in the proofs, $\seq{\epsilon_i,\delta_i}{i\ge 1}$ may also denote a subsequence of the original sequence.

\subsubsection{Two technical results}

Two additional technical results are required for the proof of the main results of this article.

First, let us state a tightness result.

\begin{proposition} \label{prop:tightness}
	Let Assumptions \ref{ass:coefficients} to \ref{ass:f0} be satisfied.

	The family of processes $\seq{\rho^{\epsilon,\delta}}{\epsilon,\delta}$ is tight in the space $\C^0_T H^{-\varsigma}_x$, for all arbitrarily small $\varsigma \in (0,1]$.
	
	Moreover, if Assumption~\ref{ass:avlemma} is also satisfied, then the family of processes $\seq{\rho^{\epsilon,\delta}}{\epsilon,\delta}$ is tight in the space $L^2_T L^2_x$.
\end{proposition}

Second, let us state a result which allows us to identify limit points of the family $\seq{\rho^{\epsilon,\delta}}{\epsilon,\delta}$.

\begin{proposition} \label{prop:limit_martingale}
	Assume that $\rho_\infty$ is a $\C^0_T H^{-\varsigma}_x$-valued random variable, such that $\rho^{\epsilon_i,\delta_i} \xrightarrow[i \to \infty]{d} \rho_\infty$, in distribution in $\C^0_T H^{-\varsigma}_x$, for some sequence $(\epsilon_i,\delta_i) \xrightarrow[i \to \infty]{} 0, \epsilon_i \in (0,\epsilon_0], \delta_i \in (0,\delta_0]$.

    Then, for all $\phi \in \Theta_{\lim}$, almost surely, for all $t\in [0,T]$, one has
	\begin{equation*}
        \phi(\rho_\infty(t)) - \phi(\rho_\infty(0)) - \int_0^t \L \phi(\rho_\infty(s)) ds = 0.
	\end{equation*}

\end{proposition}

The proofs of Propositions~\ref{prop:tightness} and~\ref{prop:limit_martingale} are technical and are given below.

\subsubsection{Proofs of the main results}
We are now in position to prove the main results of this article.

\begin{proof}[Proof of Theorem~\ref{thm:main_Hneg}]
    Let $\varsigma\in(0,1)$ be fixed. Owing to Proposition~\ref{prop:tightness}, the family of processes $\seq{\rho^{\epsilon,\delta}}{\epsilon,\delta}$ is tight in the space $\C^0_T H^{-\varsigma}_x$. Due to Prohorov Theorem, there exist a $\C^0_tT H^{-\varsigma}_x$-valued random variable $\rho_\infty$ and sequences $\seq{\epsilon_i}{i\in\mathbb{N}}$ and $\seq{\delta_i}{i\in\mathbb{N}}$, such that $\epsilon_i\to 0$ and $\delta_i\to 0$, and $\rho^{\epsilon_i,\delta_i} \xrightarrow[i \to \infty]{d} \rho_\infty$. Let us prove that $\rho_\infty$ is a weak solution of~\eqref{eq:limit_equation}: this requires to prove that~\eqref{eq:limit_equation_weak} holds and that $\rho_\infty\in L_T^\infty L_x^2$ almost surely.

To obtain~\eqref{eq:limit_equation_weak}, it suffices to use Proposition~\ref{prop:limit_martingale} and a localization argument. Introduce the auxiliary functions $\chi^r$, for all $r\ge 0$, such that $\chi^r \in \C^3_b(\R)$, $\chi^r$ is odd and
	\begin{gather*}
		\forall u \in [0,r], \chi^r(u) = u,\\
		\forall u \in [r+1,\infty), \chi^r(u) = r+1.
	\end{gather*}
Let $\xi \in H^2_x$ and define the test function $\phi^r\in \Theta_{\lim}$ by
	\begin{equation*}
		\phi^r(\rho) = \chi^r\lrpar{\lrpar{\rho,\xi}_{L^2_x}}.
	\end{equation*}
	Since $\rho_\infty \in \C^0_T H^{-\varsigma}_x$ almost surely, the random variable
	\begin{equation*}
		S \doteq \sup_{t \in [0,T]} \abs{\lrpar{\rho_\infty(t),\xi}_{L^2_x}},
	\end{equation*}
	is finite almost surely. On the one hand, by the definition of $S$, one has $\phi^S(\rho_\infty(t)) = \lrpar{\rho_\infty(t),\xi}_{L^2_x}$ and similarly $\L\phi^S(\rho_\infty(t)) =\lrpar{\bar\rho_\infty(t),\div(K \nabla \xi) + J \cdot \nabla \xi - \bar \sigma \xi}_{L^2_x}$, for all $t\ge 0$. On the other hand, owing to Proposition~\ref{prop:limit_martingale}, one has, for all $t\ge 0$,
	\begin{equation*}
	    \phi^S(\rho_\infty(t)) = \phi^S(\rho_\infty(0)) + \int_0^t \L \phi^S(\rho_\infty(s)) ds.
	\end{equation*}
Combining the arguments proves that~\eqref{eq:limit_equation_weak} holds, for all $\xi\in H^2_x$.

Let us now prove that $\rho_\infty \in L_T^\infty L^2_x$. Consider the self-adjoint operator $\mathcal S$ on $L^2_x$ defined by
\begin{equation*}
	D(\mathcal S) = H^2_x, \quad \mathcal S \rho = \div(K \nabla \rho) - \bar \sigma \rho.
\end{equation*}
Since $K$ is a positive symmetric matrix, the operator $\mathcal S - \Lambda \id$ is invertible when $\Lambda > \norm{\bar \sigma}_{\C^0_x}$. Its inverse is a compact operator, owing to the compact embedding $H^2_x \subset L^2_x$. Therefore, there exists a complete orthonormal system $\seq{e_i}{i \in \mathbb N_0}$ of $L^2_x$ composed of eigenvectors for $\lrpar{\mathcal S - \Lambda \id}^{-1}$. For $i \in \mathbb N_0$, let $\lambda_i \in \R$ be the eigenvalue of $\mathcal S$ associated to the eigenvector $e_i$: $\mathcal S e_i = \lambda_i$. Note that $\lambda_i \leq \Lambda$ for all $i\in\mathbb N_0$.

Fix $i \in \mathbb N_0$. Since $e_i \in H^2_x$, \eqref{eq:limit_equation_weak} reads for $t \in [0,T]$
\begin{equation*}
	\lrpar{\rho_\infty(t),e_i}_{L^2_x} = \lrpar{\bar\rho_0,e_i}_{L^2_x} + \lambda_i \int_0^t \lrpar{\rho_\infty(s),e_i}_{L^2_x} ds + \int_0^t \lrpar{\rho_\infty(s),J \cdot \nabla_x e_i}_{L^2_x} ds.
\end{equation*}
Therefore, one has, for all $t \in [0,T]$
\begin{equation*}
	\lrpar{\rho_\infty(t,\cdot + tJ),e_i}_{L^2_x} = e^{\lambda_i t} \lrpar{\bar\rho_0,e_i}_{L^2_x}.
\end{equation*}
Since $\lambda_i t \leq \Lambda T$ and $\bar\rho_0 \in L^2_x$, $\abs{\lrpar{\rho_\infty(t,\cdot + tJ),e_i}_{L^2_x}}^2$ is summable. We thus have $\rho_\infty(t) \in L^2_x$ and
\begin{equation}\label{eq:limit_solution_norm}
	\norm{\rho_\infty(t)}_{L^2_x}^{2}=\norm{\rho_\infty(t,\cdot + tJ)}_{L^2_x}^2 \leq e^{2 \Lambda T} \norm{\bar\rho_0}_{L^2_x}^2.
\end{equation}
This proves that $\rho_\infty \in L^\infty_T L^2_x$, and that $\rho_\infty$ is a weak solution of~\eqref{eq:limit_equation} in the sense of Definition \ref{def:limit_equation}.

To prove the convergence in distribution stated in Theorem~\ref{thm:main_Hneg}, it only remains to prove that the weak solution of~\eqref{eq:limit_equation} is unique. Since the evolution equation is linear, it is sufficient to prove that, if $\bar\rho_0 = 0$, then any weak solution $\bar\rho$ of~\eqref{eq:limit_equation} satisfies $\bar\rho(t) = 0$ for all $t \in [0,T]$. This claim is a straightforward consequence of~\eqref{eq:limit_solution_norm} (satisfied by any function $\bar\rho$ satisfying~\eqref{eq:limit_equation_weak}).
    
As a consequence, any limit point $\rho_\infty$ of the tight family $\seq{\rho^{\epsilon,\delta}}{\epsilon,\delta}$ in the space $\C^0_T H^{-\varsigma}_x$ is the unique solution $\bar\rho$ of~\eqref{eq:limit_equation}. Therefore $\rho^{\epsilon,\delta}$ converges in distribution to the unique weak solution $\bar\rho$ of~\eqref{eq:limit_equation}.
    
When the initial condition $\bar\rho_0$ is deterministic, the solution $\bar\rho$ of~\eqref{eq:limit_equation} is also deterministic. Then, using Portmanteau Theorem, in that case $\seq{\rho^{\epsilon,\delta}}{\epsilon,\delta}$ converges to $\bar\rho$ in probability.

This concludes the proof of Theorem~\ref{thm:main_Hneg}.

\end{proof}

\begin{proof}[Proof of Theorem~\ref{thm:main_L2}]
    Let Assumption~\ref{ass:avlemma} be satisfied. Owing to Proposition~\ref{prop:tightness}, the family of processes $\seq{\rho^{\epsilon,\delta}}{\epsilon,\delta}$ is tight in $L^2_T L^2_x$. Therefore, there exist a $L^2_T L^2_x$-valued random variable $\rho_\infty$ and sequences $\seq{\epsilon_i}{i\in\mathbb{N}}$ and $\seq{\delta_i}{i\in\mathbb{N}}$, such that $\epsilon_i\to 0$ and $\delta_i\to 0$, and $\rho^{\epsilon_i,\delta_i} \xrightarrow[i \to \infty]{d} \rho_\infty$. On the one hand, convergence in $L^2_T L^2_x$ implies convergence in $L^2_T H^{-\varsigma}_x$. On the other hand, owing to Theorem~\ref{thm:main_Hneg}, $\rho^{\epsilon_i,\delta_i}$ converges to $\bar\rho$ in $\C^0_T H^{-\varsigma}_x$, thus in $L^2_T H^{-\sigma}_x$. Therefore $\rho_\infty$ is equal to $\bar\rho$ in distribution. By uniqueness of the limit points, one thus obtains the convergence of $\seq{\rho^{\epsilon,\delta}}{\epsilon,\delta}$ to $\bar\rho$ in $L^2_T L^2_x$.

    It remains to establish the convergence of $f^{\epsilon,\delta}$ to $\bar\rho \M$.
    
    Note first that the mapping $h\in L^2_T L^2_x \mapsto h \M \in L^2_T \fSet$ is continuous (it is a bounded linear operator). Thus, $\rho^{\epsilon,\delta} \M \xrightarrow[\epsilon,\delta \to 0]{} \bar\rho \M$ in distribution in $L^2_T \fSet$.
    
    Owing to Slutsky's Lemma (since \cite[Theorem 4.1]{billingsley1999convergence}) and to the identity
    \begin{equation*}
    	f^{\epsilon,\delta} = - L f^{\epsilon,\delta} + \rho^{\epsilon,\delta} \M,
    \end{equation*}
    it only remains to prove that $Lf^{\epsilon,\delta} \xrightarrow[\epsilon,\delta \to 0]{} 0$ in probability in $L^2_T \fSet$.
    
    Owing to Proposition \ref{prop:L2_bound}, almost surely one has
    \begin{equation*}
    	\int_0^{T \wedge \tau^\delta} \norm{Lf^{\epsilon,\delta}(s)}_\fSet^2 ds \leq \epsilon^2 C(T) \norm{f^{\epsilon,\delta}_0}_\fSet^2.
    \end{equation*}
    Therefore, on the event $\tau^\delta \geq T$, one has $\norm{Lf^{\epsilon,\delta}}_{L^2_T \fSet} \leq \epsilon \sqrt{C(T)} \norm{f^{\epsilon,\delta}_0}_\fSet$. As a consequence, we have for $\eta \in (0,1)$
    \begin{align*}
    	\proba{\norm{Lf^{\epsilon,\delta}}_{L^2_T \fSet} > \eta}
    		&= \proba{\tau^\delta < T, \norm{Lf^{\epsilon,\delta}}_{L^2_T \fSet} > \eta}\\
    		&\phantom{=}\ + \proba{\tau^\delta \geq T, \norm{Lf^{\epsilon,\delta}}_{L^2_T \fSet} > \eta}\\
    		&\leq \proba{\tau^\delta < T} + \proba{\epsilon \sqrt{C(T)} \norm{f^{\epsilon,\delta}_0}_\fSet > \eta}\\
    		&\leq \proba{\tau^\delta < T} + \eta^{-1} \epsilon \sqrt{C(T)} \esp{\norm{f^{\epsilon,\delta}_0}_\fSet},
    \end{align*}
    owing to the Markov inequality. Using Proposition~ \ref{prop:taue_to_infty} and Assumption~\ref{ass:f0}, one obtains the convergence of $Lf^{\epsilon,\delta}$ to $0$ in probability in $L^2_T \fSet$ when $(\epsilon,\delta) \to (0,0)$. This concludes the proof of the convergence in distribution of $f^{\epsilon,\delta}$ to $\bar\rho \M$.
    
    When $\bar\rho_0$ is deterministic, $\bar\rho$ is deterministic, and the convergence results hold in probability, using Portmanteau Theorem.

This concludes the proof of Theorem~\ref{thm:main_L2}.

\end{proof}

\subsubsection{Proofs of the two technical results}

The following lemma is used in the proofs of Propositions~\ref{prop:tightness} and~\ref{prop:limit_martingale}.
\begin{lemma} \label{lemma:integral_bound}
	Let $\phi : \fSet \times E \to \R$ be a function such that
	\begin{equation*}
		\sup_{f \in \fSet,\n \in E} \frac{\abs{\phi(f,\n)}}{(1 + \norm{f}_\fSet^6) (1 + \norm{\n}_E)} < \infty.
	\end{equation*}
	Then, for all $T \in (0,\infty)$, $\epsilon \in (0,\epsilon_0]$ and $\delta \in (0,\delta_0]$ and for all random times $\tau_1$ and $\tau_2$ satisfying almost surely $0 \leq \tau_1 \leq \tau_2 \leq T$, one has
	\begin{equation} \label{eq:integral_bound_random_time}
		\sup_{\epsilon,\delta}\esp{\int_{\tau_1 \wedge \tau^\delta}^{\tau_2 \wedge \tau^\delta} \abs{\phi(f\edstopped(t),m\dstopped(t))} dt} \lesssim \esp{\tau_2 - \tau_1}^{\frac12-\frac1\gamma}.
	\end{equation}

\end{lemma}

Recall that $\gamma>2$ is given by Assumption~\ref{ass:moments_m_gamma}.

\begin{proof}[Proof of Lemma~\ref{lemma:integral_bound}]

Since $m\dstopped(t) = m^\delta(t)$ for all $t \in [\tau_1 \wedge \tau^\delta, \tau_2 \wedge \tau^\delta]$, the estimate~\eqref{eq:L2_bound} from Proposition \ref{prop:L2_bound} yields
	\begin{equation*}
\esp{\int_{\tau_1 \wedge \tau^\delta}^{\tau_2 \wedge \tau^\delta} \abs{\phi(f\edstopped(t),m\dstopped(t))} dt} \lesssim_{T,\phi} \esp{\int_{\tau_1 \wedge \tau^\delta}^{\tau_2 \wedge \tau^\delta} \lrpar{1 + \norm{f^{\epsilon,\delta}_0}_\fSet^6} \lrpar{1 + \norm{m^\delta(t)}} dt}.
	\end{equation*}
	
Let $p^* \in [1,\infty)$ defined by $\frac{1}{\gamma} + \frac{1}{2} + \frac{1}{p^*} = 1$. Using the moment estimates for $m^{\delta}$ (see Assumption~\ref{ass:moments_m_gamma}) and for $f^{\epsilon,\delta}$ (see Assumption~\ref{ass:f0}), and applying H\"older inequality, one obtains
	\begin{align*}
&\esp{\int_{\tau_1 \wedge \tau^\delta}^{\tau_2 \wedge \tau^\delta} \abs{\phi(f\edstopped(t),m\dstopped(t))} dt}\\
			&\lesssim \int_0^T \esp{\ind{[\tau_1,\tau_2]}(t) \lrpar{1 + \norm{f^{\epsilon,\delta}_0}_\fSet^6} \lrpar{1 + \norm{m^\delta(t)}}} dt\\
			&\lesssim \int_0^T \esp{\ind{[\tau_1,\tau_2]}(t)^{p^*}}^{1/p^*} \esp{\lrpar{1 + \norm{f^{\epsilon,\delta}_0}_\fSet^6}^2}^{1/2} \esp{\lrpar{1 + \norm{m^\delta(t)}}^\gamma}^{1/\gamma} dt\\
			&\lesssim \int_0^T \esp{\ind{[\tau_1,\tau_2]}(t)}^{1/p^*} dt\\
			&\lesssim \lrpar{\int_0^T \esp{\ind{[\tau_1,\tau_2]}(t)} dt}^{1/p^*}\\
			&\lesssim \esp{\int_{\tau_1}^{\tau_2} dt}^{1/p^*}\\
			&\lesssim \esp{\tau_2 - \tau_1}^{1/p^*}.
	\end{align*}
All the estimates above are uniform with respect to $\epsilon\in(0,\epsilon_0]$ and $\delta\in(0,\delta_0]$. This concludes the proof of Lemma~\ref{lemma:integral_bound}.	
\end{proof}

Before proceeding, let us recall some useful results concerning tightness.

Introduce the Skorokhod space $D_T H^{-\varsigma}_x$, which is the space of $H^{-\varsigma}_x$-valued càdlàg functions on $[0,T]$. For all $X \in \C^0_T H^{-\varsigma}_x$ and all $\Delta\in[0,T]$, set
	\begin{gather*}
		w_X(\Delta) \doteq \sup_{0 \leq t \leq s \leq t + \Delta \leq T} \norm{X(s)-X(t)}\\
		w'_X(\Delta) \doteq \sup_{\seq{t_i}{i}} \max_i \sup_{t_i \leq t \leq s < t_{i+1}} \norm{X(s)-X(t)},
	\end{gather*}
	where $\seq{t_i}{i}$ denotes any finite subdivision of $[0,T]$. The moduli of continuity $\omega_X$ and $\omega_X'$ satisfy the following inequality (see~\cite[equation (14.11)]{billingsley1999convergence}): for all $\Delta \in [0,T]$ and all $X \in \C^0_T H^{-\varsigma}_x$, one has
	\begin{equation*}
		w_X(\Delta) \leq 2 w'_X(\Delta).
	\end{equation*}
We refer to~\cite[Theorems 8.2 and 15.2]{billingsley1999convergence} for tightness criteria in the spaces $\C^0_T H^{-\varsigma}_x$ and $D_T H^{-\varsigma}$. As a consequence, tightness in $D_T H^{-\varsigma}_x$ of a family a family of processes $\seq{X^{\epsilon,\delta}}{\epsilon,\delta}$ implies its tightness in $\C^0_T H^{-\varsigma}_x$.

	Observe that tightness in $D_T H^{-\varsigma}$ is easier to prove than tightness in $\C^0_T H^{-\varsigma}_x$, owing to \cite[Theorem 3.1]{jakubowski1986on}. More precisely, since the class of functions $\Theta_{\lim}$ is closed under addition and separates points, tightness of a family $\seq{X^{\epsilon,\delta}}{\epsilon,\delta}$ in the Skorokhod space $D_T H^{-\varsigma}$ is equivalent to the following claims:
	\begin{enumerate}[label=(\roman*)]
		\item \label{item:first_claim} For all $\eta \in (0,1]$, there exists a compact set $K_\eta \subset H^{-\varsigma}_x$ such that, for all $\epsilon \in (0,\epsilon_0]$ and $\delta \in (0,\delta_0]$,
		\begin{equation*} 
			\proba{\forall t \in [0,T],X^{\epsilon,\delta}(t) \in K_\eta} > 1 - \eta.
		\end{equation*}
		\item \label{item:second_claim} For all $\phi \in \Theta_{\lim}$, $\seq{\phi(X^{\epsilon,\delta})}{\epsilon,\delta}$ is tight in the Skorohod space $D([0,T],\R)$ of càdlàg real-valued functions defined on the interval $[0,T]$.
	\end{enumerate}
To check that $(ii)$ is satisfied, we employ Aldous's criterion, see~\cite[Theorem 4.5 p356]{jacod2003limit} (note that the criterion is simplified since here $\phi$ is bounded): it suffices to prove that for all $\eta \in (0,\infty)$, one has 
\begin{equation} \label{eq:aldous}
	 \lim_{\Delta \to 0} \limsup_{(\epsilon,\delta) \to (0,0)} \sup_{\tau_1 \leq \tau_2 \leq \tau_1 + \Delta} \proba{\abs{\phi(X^{\epsilon,\delta}(\tau_2)) - \phi(X^{\epsilon,\delta}(\tau_1))} > \eta} = 0,
\end{equation}
where $\sup_{\tau_1 \leq \tau_2 \leq \tau_1 + \Delta}$ denotes the supremum with respect to all $\seq{\mathcal F^\delta_t}{t \in \R^+}$-stopping times $\tau_1$ and $\tau_2$ satisfying a.s. $\tau_1, \tau_2 \in [0,T]$ and $\tau_1 \leq \tau_2 \leq \tau_1 + \Delta$.

\begin{proof}[Proof of Proposition~\ref{prop:tightness}]
First, recall that $\tau^\delta\to\infty$ in probability, when $\delta\to 0$ (see Proposition~\ref{prop:taue_to_infty}. Owing to Slutsky's Lemma \cite[Theorem 4.1]{billingsley1999convergence}, it thus suffices to prove the tightness of the family $\seq{\rho\edstopped}{\epsilon,\delta}$.

Let us first establish the tightness in the space $\C^0_T H^{-\varsigma}_x$, with an arbitrarily small parameter $\varsigma>0$. As explained above, in fact we establish the tightness in $D_T H^{-\varsigma}$, using the criteria stated above. First, \ref{item:first_claim} is satisfied: indeed the embedding $L^2_x \subset H^{-\varsigma}_x$ is compact, and the a priori estimate~\eqref{eq:L2_bound} (see Proposition~\ref{prop:L2_bound}) yields the uniform moment bound
\[
\underset{\epsilon,\delta}\sup~\E[\underset{0\le t\le T}\sup~\norm{\rho\edstopped(t)}_{L^2_x}^2]\leq\underset{\epsilon,\delta}\sup~\E[\underset{0\le t\le T}\sup~\norm{f\edstopped(t)}_{\fSet}^2]\leq \underset{\epsilon,\delta}\sup~\E[\norm{f^{\epsilon,\delta}(0)}_{L^2_x}^2]<\infty
\]
owing to Assumption~\ref{ass:f0}. Then~\ref{item:first_claim} is a straightforward consequence of Markov inequality.

It remains to establish~\ref{item:second_claim}, using Aldous criterion~\ref{eq:aldous}. Let $\phi^{\epsilon,\delta}$ be the perturbed test function given by~\eqref{eq:def_phi_epsilon}, see Proposition~\ref{prop:corrector}, and let $M^{\epsilon,\delta}_{\phi^{\epsilon,\delta}}$ be defined as by~\eqref{eq:martingale_formulation_epsilon} (in Proposition~\ref{prop:martingale_formulation_epsilon}). For all $t\ge 0$, set
\begin{align}
		\theta^{\epsilon,\delta}(t)
			&= \phi(\rho^{\epsilon,\delta}(0)) + \phi^{\epsilon,\delta}(f^{\epsilon,\delta}(t),m^\delta(t)) - \phi^{\epsilon,\delta}(f^{\epsilon,\delta}(0),m^\delta(0)) \nonumber\\
			&= \phi(\rho^{\epsilon,\delta}(0)) + \int_0^t \L^{\epsilon,\delta} \phi^{\epsilon,\delta} (f^{\epsilon,\delta}(s),m^\delta(s))ds + M^{\epsilon,\delta}_{\phi^{\epsilon,\delta}}(t), \label{eq:expr_thetae}
	\end{align}
For all stopping times $\tau_1,\tau_2$, one has the equality	
\begin{equation*}
	\begin{split}
	\phi(\rho\edstopped(\tau_2)) - \phi(\rho\edstopped(\tau_1))
		&= \lrpar{\theta\edstopped(\tau_2) - \theta\edstopped(\tau_1)}\\
		&\phantom{=}\ - \lrpar{\phi^{\epsilon,\delta}(f\edstopped(\tau_2),m\dstopped(\tau_2)) - \phi(\rho\edstopped(\tau_2))}\\
		&\phantom{=}\ + \lrpar{\phi^{\epsilon,\delta}(f\edstopped(\tau_1),m\dstopped(\tau_1)) - \phi(\rho\edstopped(\tau_1))}.
	\end{split}
\end{equation*}
On the one hand, owing to the error estimate~\eqref{eq:control_error_phi}, the estimates~\eqref{eq:taue_bounds_m} and~\eqref{eq:L2_bound} for $\norm{m\dstopped(t)}_E$ and $\norm{f\edstopped(t)}_\fSet$ yield
\begin{equation*}
	\abs{\phi^{\epsilon,\delta}(f\edstopped(t),m\dstopped(t)) - \phi(\rho\edstopped(t))} \lesssim_{\phi} (1 + \norm{f^{\epsilon,\delta}_0}_\fSet^2) (\epsilon + \delta^2 (1 + \delta^{-\alpha})).
\end{equation*}
Since $\alpha < 1$ in Definition \ref{def:taue}, we get
\begin{equation*}
	\esp{\abs{\phi^{\epsilon,\delta}(f\edstopped(\tau_i),m\dstopped(\tau_i)) - \phi(\rho\edstopped(\tau_i))}} \xrightarrow[\epsilon,\delta \to 0]{} 0,
\end{equation*}
for $i=1,2$.

On the other hand, we claim that
\begin{equation} \label{eq:aim_aldous2}
	\sup_{\epsilon,\delta} \sup_{\tau_1 \leq \tau_2 \leq \tau_1 + \Delta} \esp{\abs{\theta\edstopped(\tau_2) - \theta\edstopped(\tau_1)}} \xrightarrow[\Delta \to 0]{} 0.
\end{equation}
To prove that this claim holds, note that
\begin{equation} \label{eq:temp_new_aim_final_cut}
	\abs{\theta\edstopped(\tau_2) - \theta\edstopped(\tau_1)} \leq \int_{\tau_1 \wedge \tau^\delta}^{\tau_2 \wedge \tau^\delta} \abs{\L^{\epsilon,\delta} \phi^{\epsilon,\delta}(f^{\epsilon,\delta}(s),m^\delta(s))} ds + \abs{M\edstopped_{\phi^{\epsilon,\delta}}(\tau_2) - M\edstopped_{\phi^{\epsilon,\delta}}(\tau_1)}.
\end{equation}
To treat the first term on the right-hand side of~\eqref{eq:temp_new_aim_final_cut}, observe that for $(f,\n) \in \fSet \times E$, one has
\begin{align*}
		\abs{\L^{\epsilon,\delta} \phi^{\epsilon,\delta} (f,\n)}
			&\lesssim \epsilon \lrpar{1 + \norm{f}_\fSet^3} \lrpar{1 + \norm{\n}} + \delta^2 \lrpar{1 + \norm{f}_\fSet^3} \lrpar{1 + \norm{\n}^2}\\
			&\phantom{=}\ + \abs{\L \phi(\rho)}\\ 
			&\lesssim \lrpar{1 + \norm{f}_\fSet^3} \lrpar{1 + \norm{\n}} + \delta^2 \lrpar{1 + \norm{f}_\fSet^3} \lrpar{1 + \norm{\n}^2},
	\end{align*}
owing to the error estimate~\eqref{eq:control_perturbed_generator} for $\L^{\epsilon,\delta}\phi^{\epsilon,\delta}-\L\phi$ (see Proposition~\ref{prop:corrector}) and to the expression~\eqref{eq:def_L} of $\L\phi$. Using Lemma~\ref{lemma:integral_bound}, one obtains 
\begin{equation*}
		\sup_{\epsilon,\delta} \sup_{\tau_1 \leq \tau_2 \leq \tau_1 + \Delta} \esp{\int_{\tau_1 \wedge \tau^\delta}^{\tau_2 \wedge \tau^\delta} \abs{\lrpar{1 + \norm{f^{\epsilon,\delta}(s)}_\fSet^3} \lrpar{1 + \norm{m^\delta(s)}}} ds} \lesssim \Delta^{1/2-1/\gamma} \xrightarrow[\Delta \to 0]{} 0,
\end{equation*}
since $\gamma>2$ (see Assumption~\ref{ass:moments_m_gamma}. In addition, recall that $\alpha<1$: using the estimate~\eqref{eq:taue_bounds_m}, one obtains
\begin{equation*}
		\sup_{\epsilon,\delta} \sup_{\tau_1 \leq \tau_2 \leq \tau_1 + \Delta} \esp{\int_{\tau_1 \wedge \tau^\delta}^{\tau_2 \wedge \tau^\delta} \abs{\delta^2 \lrpar{1 + \norm{f^{\epsilon,\delta}(s)}_\fSet^3} \lrpar{1 + \norm{m^\delta(s)}^2}} ds} \lesssim \Delta \xrightarrow[\Delta \to 0]{} 0.
\end{equation*}

To treat the second term on the right-hand side of~\eqref{eq:temp_new_aim_final_cut}, note that $M\edstopped_{\phi^{\epsilon,\delta}}$ is a square-integrable martingale, owing to Proposition~\ref{prop:martingale_formulation_epsilon}. In addition, since $\tau_1\leq \tau_2$ are stopping times, one has
	\begin{multline*}
		\esp{\abs{M\edstopped_{\phi^{\epsilon,\delta}}(\tau_2) - M\edstopped_{\phi^{\epsilon,\delta}}(\tau_1)}^2}
			= \esp{\abs{M\edstopped_{\phi^{\epsilon,\delta}}(\tau_2)}^2 - \abs{M\edstopped_{\phi^{\epsilon,\delta}}(\tau_1)}^2}\\
			= \frac{1}{\delta^2} \esp{\int_{\tau_1 \wedge \tau^\delta}^{\tau_2 \wedge \tau^\delta} \lrpar{\L_m((\phi^{\epsilon,\delta})^2) - 2 \phi^{\epsilon,\delta} \L_m \phi^{\epsilon,\delta}}(f^{\epsilon,\delta}(s),m^\delta(s)) ds}. 
	\end{multline*}
Recall the expression~\eqref{eq:def_phi_epsilon} of the perturbed test function $\phi^{\epsilon,\delta}$. Since $\phi$, $\phi_{1,0}$ and $\phi_{2,0}$ do not depend on $\n\in E$, one has
\begin{align*}
		&\esp{\abs{M\edstopped_{\phi^{\epsilon,\delta}}(\tau_2) - M\edstopped_{\phi^{\epsilon,\delta}}(\tau_1)}^2}\\
		&= \delta^2 \esp{\int_{\tau_1 \wedge \tau^\delta}^{\tau_2 \wedge \tau^\delta} \lrpar{\L_m((\phi_{0,2} + \epsilon \phi_{1,2})^2) - 2 (\phi_{0,2} + \epsilon \phi_{1,2}) \L_m (\phi_{0,2} + \epsilon \phi_{1,2})}(f^{\epsilon,\delta}(s),m^\delta(s)) ds}.
	\end{align*}
The correctors $\phi_{0,2}$ and $\phi_{1,2}$ satisfy the properties~\eqref{eq:control_phi_02} and~\eqref{eq:control_phi_12} respectively. In addition, $f^{\epsilon,\delta}$ and $m^{\delta}$ satisfy the estimates~\eqref{eq:L2_bound} and~\eqref{eq:taue_bounds_m} respectively. Using Assumption~\ref{ass:regular_Lm} and the condition $0\le \tau_2-\tau_1\le \Delta$, one finally obtains
\begin{align*}
		\esp{\abs{M\edstopped_{\phi^{\epsilon,\delta}}(\tau_2) - M\edstopped_{\phi^{\epsilon,\delta}}(\tau_1)}^2} &\lesssim \Delta \delta^2 \esp{(1 + \norm{f_0^{\epsilon,\delta}}_\fSet^2)(1 + \delta^{-2\alpha})}\\
		& \lesssim \Delta \xrightarrow[\Delta \to 0]{} 0,
	\end{align*}
since $\alpha<1$ and using Assumption~\ref{ass:f0}.

Gathering the estimates for the two terms on the right-hand side of~\eqref{eq:temp_new_aim_final_cut}, the claim~\eqref{eq:aim_aldous2} is proved. One finally obtains
\[
\sup_{\epsilon,\delta} \sup_{\tau_1 \leq \tau_2 \leq \tau_1 + \Delta} \esp{\abs{\phi(\rho\edstopped(\tau_2)) - \phi(\rho\edstopped(\tau_1))}} \xrightarrow[\Delta \to 0]{} 0,
\]
hence~\eqref{eq:aldous} holds. This concludes the proof of the tightness of the family $\seq{\rho\edstopped}{\epsilon,\delta}$ is tight in the space $\C^0_T H^{-\varsigma}_x$.

It remains to prove the tightness of $\seq{\rho\edstopped}{\epsilon,\delta}$ is tight in the space $L^2_T L^2_x$, if Assumption~\ref{ass:avlemma} is satisfied. It suffices to establish the following claims: for all $\eta\in(0,1)$, there exists $R\in(0,\infty)$ and $\varsigma''\in(0,1)$, such that
\begin{equation} \label{eq:strong_tightness_claim2}
	\lim_{\Delta \to 0} \limsup_{(\epsilon,\delta) \to (0,0)} \proba{w_{\rho\edstopped}(\Delta) > \eta} = 0,
\end{equation}
and
\begin{equation} \label{eq:strong_tightness_claim1}
	\sup_{\epsilon,\delta} \proba{\norm{\rho\edstopped}_{L^2_T H^{\varsigma''}_x} > R} < \eta.
\end{equation}
Indeed, for all $R>0$, ${\varsigma''}>0$ and $\eta : (0,\infty) \to [0,\infty)$ such that $\eta(\Delta) \xrightarrow[\Delta \to 0]{} 0$, the set
\begin{equation*}
	K_{R,\eta} \doteq \set{\rho \in L^2_T L^2_x \mid \norm{\rho}_{L^2_T H^{\varsigma''}_x} \leq R \mbox{ and } \forall \Delta \in (0,1), w_\rho(\Delta) < \eta(\Delta)}
\end{equation*}
is compact in $L^2_T L^2_x$, see~\cite[Theorem 5]{simon1987compact}

The first claim is a consequence of the tightness of $\seq{\rho\edstopped}{\epsilon,\delta}$ in $\C^0_TH_x^{-\varsigma}$ proved above, see~\cite[Theorem 8.2]{billingsley1999convergence}.

Owing to Markov inequality, it suffices to prove the following inequality: 
\begin{equation} \label{eq:strong_tightness_claim1bis}
	\sup_{\epsilon,\delta} \esp{\norm{\rho\edstopped}_{L^2_T H^{\varsigma''}_x}} \lesssim 1,
\end{equation}

Let $g^{\epsilon,\delta} = \epsilon \partial_t f^{\epsilon,\delta} + a(v) \cdot \nabla_x f^{\epsilon,\delta} + \epsilon b(v) \cdot \nabla_x f^{\epsilon,\delta}$. By Assumption~\ref{ass:avlemma}, we are in position to apply an averaging lemma , precisely \cite[Theorem 2.3]{bouchut1999averaging} (with $f(t) = f^{\epsilon,\delta}(\epsilon t)$, $g(t) = g^{\epsilon,\delta}(\epsilon t)$ and $h = 0$ until time $T \wedge \tau^\delta$). After rescaling the time $t \mapsto t/\epsilon$, one obtains the inequality
\begin{align*}
	\norm{\rho\edstopped}_{L^2_T H^{\varsigma'/4}_x}^2
		&= \int_0^{T \wedge \tau^\delta} \norm{\rho\edstopped(t)}_{H^{\varsigma'/4}_x}^2 dt\\
		&\lesssim \epsilon \norm{f_0^{\epsilon,\delta}}_{L^2_x}^2 + \int_0^{T \wedge \tau^\delta} \norm{f\edstopped(t)}_{\fSet}^2 dt + \int_0^{T \wedge \tau^\delta}\norm{g\edstopped(t)}_{\fSet}^2 dt.
\end{align*}
Applying the Cauchy-Schwarz inequality gives
\begin{align*}
	\norm{g\edstopped(t)}_{\fSet}
		&= \norm{\epsilon f\edstopped(t) \sigma(m\dstopped(t)) + \frac{1}{\epsilon} Lf\edstopped(t)}_{\fSet}\\
		&\leq \epsilon \norm{f\edstopped(t)}_{\fSet} \norm{\sigma(m\dstopped(t))}_{\C^0_x} + \frac{1}{\epsilon} \norm{Lf\edstopped(t)}_{\fSet}.
\end{align*}
It is now crucial to use the a priori estimate~\eqref{eq:L2_bound} (see Proposition~\ref{prop:L2_bound}) to control the intergral term
\[
\int_{0}^{T\wedge\tau^\delta}\frac{1}{\epsilon} \norm{Lf\edstopped(t)}_{\fSet} dt
\]
uniformly with respect to $\epsilon,\delta$. Using Assumption \ref{ass:f0}, the estimates~\eqref{eq:taue_bounds_m} and~\eqref{eq:L2_bound} for $m^{\delta,\tau^\delta}(t)$ and $f\edstopped(t)$, and Lemma~\ref{lemma:integral_bound}, the claim~\eqref{eq:strong_tightness_claim1bis} is proved, with $\varsigma'' = \frac{\varsigma'}{4}$.

Since we proved~\eqref{eq:strong_tightness_claim2} and~\eqref{eq:strong_tightness_claim1}, the family $\seq{\rho\edstopped}{\epsilon,\delta}$ is tight in $L^2_T L^2_x$, when Assumption~\ref{ass:avlemma} is satisfied.

This concludes the proof of Proposition~\ref{prop:tightness}.
\end{proof}

\begin{proof}[Proof of Proposition~\ref{prop:limit_martingale}]
Let $\rho_\infty$ be such that $\rho_{\infty}=\underset{i\to\infty}\lim~\rho^{\epsilon_i,\delta_i}$, for some sequence $(\epsilon_i,\delta_i)\to (0,0)$.

For all $\phi \in \Theta_{\lim}$, define the stochastic process $M_\phi$ as follows: for all $t\ge 0$,
\begin{equation*}
    M_\phi(t) = \phi(\rho_\infty(t)) - \phi(\rho_\infty(0)) - \int_0^t \L \phi(\rho_\infty(s)) ds.
\end{equation*}
Let us start by proving that, for all $\phi \in \Theta_{\lim}$, $M_\phi$ is a square integrable martingale adapted to the filtration generated by $\rho_\infty$.

Let the test function $\phi \in \Theta_{\lim}$ be fixed. Since $\phi$ is bounded, to prove that the process $M_\phi$ is square-integrable, it suffices to prove that
\begin{equation} \label{eq:Lphi_moments}
    \sup_{t \in [0,T]} \esp{\abs{\L\phi(\rho_\infty(t))}^2} < \infty.
\end{equation}
Observe that, for all $t \in [0,T]$, the mapping $\rho \in \C^0_T H^{-\varsigma}_x\mapsto \abs{\L\phi(\rho(t))}^2\in \mathbb{R}$ is continuous. Thus one has the convergence in distribution $\abs{\L\phi(\rho\edstoppedi(t))}^2 \xrightarrow[i \to \infty]{d} \abs{\L\phi(\rho_\infty(t))}^2$ (see~\cite[Proposition IX.5.7]{bourbaki2004integrationII}). In addition, $\seq{\abs{\L\phi(\rho\edstoppedi(t))}^2}{\epsilon_i,\delta_i}$ is uniformly integrable: one has
\begin{equation*}
	\sup_{i\in\mathbb N} \esp{\abs{\L\phi(\rho\edstoppedi(t))}^4} \lesssim \sup_{\epsilon,\delta} \esp{\norm{f^{\epsilon,\delta}_0}_\fSet^4} < \infty,
\end{equation*}
using the estimates from Proposition~\ref{prop:L2_bound} and Assumption~\ref{ass:f0}, and the expression~\eqref{eq:Lrho-theta} of $\L\phi$ when $\phi\in \Theta_{\lim}$. Using~\cite[Theorem 5.4]{billingsley1999convergence}, the convergence in distribution and uniform integrability property give
\begin{equation*}
	\esp{\abs{\L\phi(\rho_\infty(t))}^2} = \lim_{i \to \infty} \esp{\abs{\L\phi(\rho\edstoppedi(t))}^2} \lesssim \sup_{\epsilon,\delta} \esp{\norm{f^{\epsilon,\delta}_0}_\fSet^2}<\infty.
\end{equation*}
This yields the square integrability property~\eqref{eq:Lphi_moments}.

The next step is to prove that $M_\phi$ is a martingale. Let $0 \leq s_1 \leq ... \leq s_j \leq s \leq t$ and let $g \in \C^0_b(\lrpar{H^{-\varsigma}_x}^j)$ be a continuous bounded function. Define the mapping
	\begin{equation*}
		\Phi : \rho \in \C_T^0H_x^{-\varsigma} \mapsto \lrpar{\phi(\rho(t)) - \phi(\rho(s)) - \int_s^t \L \phi(\rho(u)) du} g(\rho(s_1), ..., \rho(s_j)).
	\end{equation*}
To prove that $M_\phi$ is a martingale, it suffices to prove the following claim: $\esp{\Phi(\rho_\infty)} = 0$.

Let us first check that $\esp{\Phi(\rho\edstoppedi)}\to \esp{\Phi(\rho_\infty)}$ when $i\to\infty$. This claim follows from the definition of $\rho_\infty$ as the limit in distribution of $\rho\edstoppedi$ and straightforward arguments. The mapping $\Phi$ is continuous on $\C^0_T H^{-\varsigma}_x$ and $\seq{\Phi(\rho\edstoppedi)}{\epsilon_i,\delta_i}$ is uniformly integrable. Using the same arguments as for the proof of~\eqref{eq:Lphi_moments}, one obtains the following convergence result:
\begin{equation*}
		\esp{\Phi(\rho\edstoppedi)} \xrightarrow[i \to \infty]{} \esp{\Phi(\rho_\infty)}.
\end{equation*}

Let us now check that $\esp{\Phi(\rho\edstoppedi)}\to 0$ when $i\to\infty$. This claim follows from a martingale property and the perturbed test function to take the limit $i\to\infty$. Let $\phi^{\epsilon,\delta}$ be the perturbed test function given by~\eqref{eq:def_phi_epsilon}, see Proposition~\ref{prop:corrector}. Using the martingale property from Proposition~\ref{prop:martingale_formulation_epsilon}, one has
	\begin{multline*}
		\esp{\lrpar{\phi^{\epsilon_i,\delta_i}(f\edstoppedi(t),m\dstoppedi(t)) - \phi^{\epsilon_i,\delta_i}(f\edstoppedi(s),m\dstoppedi(s))  \vphantom{\lrpar{\phi^{\epsilon_i,\delta_i}(f\edstoppedi(t),m\dstoppedi(t)) - \phi^{\epsilon_i,\delta_i}(f\edstoppedi(s),m\dstoppedi(s)) - \int_{s \wedge \tau^{\delta_i}}^{t \wedge \tau^{\delta_i}} \L^{\epsilon_i,\delta_i} \phi^{\epsilon_i,\delta_i}(f^{\epsilon_i,\delta_i}(u),m^{\delta_i}(u))du}} \right. \right.\\
		\left. \left. \vphantom{\lrpar{\phi^{\epsilon_i,\delta_i}(f\edstoppedi(t),m\dstoppedi(t)) - \phi^{\epsilon_i,\delta_i}(f\edstoppedi(s),m\dstoppedi(s)) - \int_{s \wedge \tau^{\delta_i}}^{t \wedge \tau^{\delta_i}} \L^{\epsilon_i,\delta_i} \phi^{\epsilon_i,\delta_i}(f^{\epsilon_i,\delta_i}(u),m^{\delta_i}(u))du}} - \int_{s \wedge \tau^{\delta_i}}^{t \wedge \tau^{\delta_i}} \L^{\epsilon_i,\delta_i} \phi^{\epsilon_i,\delta_i}(f^{\epsilon_i,\delta_i}(u),m^{\delta_i}(u))du}
		g(\rho\edstoppedi(s_1), ..., \rho\edstoppedi(s_j))} = 0.
	\end{multline*}
	Using the expression $\phi^{\epsilon,\delta} = \phi + \epsilon \phi_{1,0} + \epsilon^2 \phi_{2,0} + \delta^2 \phi_{0,2} + \epsilon \delta^2 \phi_{1,2}$ (see~\eqref{eq:def_phi_epsilon}, and the boundedness of $g$, one obtains the upper bound
	\begin{equation*}
		\abs{\esp{\Phi(\rho\edstoppedi)}} \lesssim \sum_{j=1}^6 \esp{\abs{r_j}},
	\end{equation*}
	with
	\begin{gather*}
		r_1 = \epsilon_i (\phi_{1,0}(f\edstoppedi(t)) - \phi_{1,0}(f\edstoppedi(s)))\\
		r_2 = \epsilon_i^2 (\phi_{2,0}(f\edstoppedi(t)) - \phi_{2,0}(f\edstoppedi(s)))\\
		r_3 = \delta_i^2 (\phi_{0,2}(f\edstoppedi(t),m\dstoppedi(t)) - \phi_{0,2}(f\edstoppedi(s),m\dstoppedi(s)))\\
		r_4 = \epsilon_i \delta_i^2 (\phi_{1,2}(f\edstoppedi(t),m\dstoppedi(t)) - \phi_{1,2}(f\edstoppedi(s),m\dstoppedi(s)))\\
		r_5 = \int_{s \wedge \tau^{\delta_i}}^{t \wedge \tau^{\delta_i}} \lrpar{\L^{\epsilon_i,\delta_i} \phi^{\epsilon_i,\delta_i}(f\edstoppedi(u),m\dstoppedi(u)) - \L \phi (\rho\edstoppedi(u))} du\\
		r_6 = \int_{t \wedge \tau^{\delta_i}}^{t} \L \phi(\rho\edstoppedi(u)) du - \int_{s \wedge \tau^{\delta_i}}^{s} \L \phi(\rho\edstoppedi(u)) du.
	\end{gather*}
Using estimates~\eqref{eq:control_phi_10},~\eqref{eq:control_phi_20},~\eqref{eq:control_phi_02} and~\eqref{eq:control_phi_12} of Proposition \ref{prop:corrector} and estimates~\eqref{eq:taue_bounds_m} and~\eqref{eq:L2_bound} on $m\dstoppedi$ and $f\edstoppedi$, one has
	\begin{gather*}
		\abs{r_1} \lesssim \epsilon_i \norm{f^{\epsilon_i,\delta_i}_0}_\fSet,\\
		\abs{r_2} \lesssim \epsilon_i^2 (1+\norm{f^{\epsilon_i,\delta_i}_0}_\fSet^2),\\
		\abs{r_3} \lesssim \delta_i^2 (1 + \norm{f^{\epsilon_i,\delta_i}_0}_\fSet) (1 + \delta_i^{-\alpha}),\\
		\abs{r_4} \lesssim \epsilon_i \delta_i^2 (1 + \norm{f^{\epsilon_i,\delta_i}_0}_\fSet^2) (1 + \delta_i^{-\alpha}).
	\end{gather*}
Since $\alpha < 1$ (see Definition~\ref{def:taue}), using Assumption~\ref{ass:f0}, one obtains $\esp{\abs{r_k}} \xrightarrow[i \to \infty]{} 0$, for $k \in \llrrbracket{1,4}$.

To treat the next term $r_5$, it is necessary to use Lemma~\ref{lemma:integral_bound}: using the estimate~\eqref{eq:control_perturbed_generator}, one obtains
	\begin{equation*}
		\esp{\abs{r_5}} \lesssim \epsilon_i (t-s)^{\frac12-\frac1\gamma} + \delta_i^2 (1 + \delta_i^{-2\alpha}) \xrightarrow[i \to \infty]{} 0,
	\end{equation*}
using the condition $\alpha<1$.

The last term $r_6$ is treated as follows: using the Cauchy-Schwarz inequality, the expression~\eqref{eq:Lrho-theta} for $\L\phi$ and the estimate~\ref{prop:L2_bound}, one obtains
	\begin{align*}
		\esp{\abs{r_6}}^2
			&\lesssim\esp{\norm{f^{\epsilon,\delta}_0}_\fSet^2} \esp{\lrpar{\abs{t - t \wedge \tau^{\delta_i}} + \abs{s - s \wedge \tau^{\delta_i}}}^2}\\
			&\lesssim \proba{\tau^{\delta_i} < T} \xrightarrow[i \to \infty]{} 0,
	\end{align*}
owing to Assumption~\ref{ass:f0} and to Proposition~\ref{prop:taue_to_infty} which gives the convergence in probability $\tau^{\delta}\to \infty$.

Gathering the results, one obtains $\esp{\Phi(\rho\edstoppedi)}\to 0$ when $i\to\infty$, hence $\esp{\Phi(\rho_\infty)} = \lim_{i \to \infty} \esp{\Phi(\rho\edstoppedi)} = 0$. This concludes the proof that, for all $\phi \in \Theta_{\lim}$, $M_\phi$ is a square-integrable martingale adapted to the filtration generated by $\rho_\infty$.

The final step is to prove that $\esp{\abs{M_\phi(t)}^2} = 0$, for all $t\in[0,T]$ and $\phi\in\Theta_{\lim}$.

Let $\Delta \in (0,\infty)$ be an arbitrarily small real-number and let $0 = t_0 < t_1 < ... < t_j = t$ be a subdivision of $[0,t]$ such that $\max_{k \in \llrrbracket{0,j-1}} \abs{t_{k+1} - t_k} \leq \Delta$. Since $M_\phi$ is a centered martingale, with $M_\phi(0)=0$, one has
\begin{align}
    \esp{\abs{M_\phi(t)}^2}
        &= \sum_{i = 0}^{n-1} \esp{\abs{M_\phi(t_{k+1}) - M_\phi(t_k)}^2} \nonumber\\
        &\leq 2 \sum_{i = 0}^{n-1} \esp{\abs{\phi(\rho_\infty(t_{k+1})) - \phi(\rho_\infty(t_k))}^2} + 2 \sum_{i = 0}^{n-1} \esp{\abs{\int_{t_k}^{t_{k+1}} \L\phi(\rho_\infty(s)) ds}^2}. \label{eq:estimate_quadratic_variation}
\end{align}

On the one hand, for $0 \leq t \leq t' \leq T$, one has the identity
\begin{multline*}
    \abs{\phi(\rho_\infty(t')) - \phi(\rho_\infty(t))}^2 = M_{\phi^2}(t') - M_{\phi^2}(t) - 2 \phi(\rho_\infty(t)) \lrpar{M_\phi(t') - M_\phi(t)}\\
    + \int_{t}^{t'} \L(\phi^2)(\rho_\infty(s)) ds - 2 \phi(\rho_\infty(t)) \int_{t}^{t'} \L(\phi)(\rho_\infty(s)) ds.
\end{multline*}
The test functions $\phi$ and $\phi^2$ belong to the class $\Theta_{\lim}$, therefore, owing to the first part of the proof, $M_\phi$ and $M_{\phi^2}$ are centered martingales for the filtration generated by $\rho_\infty$. As a consequence,
\begin{equation} \label{eq:increment_phi_rho}
    \esp{\abs{\phi(\rho_\infty(t')) - \phi(\rho_\infty(t))}^2} = \esp{\int_{t}^{t'} \L(\phi^2)(\rho_\infty(s)) ds - 2 \phi(\rho_\infty(t)) \int_{t}^{t'} \L\phi(\rho_\infty(s)) ds}.
\end{equation}
Since $\L$ is a first order derivative operator (see~\eqref{eq:def_L}), it is straightforward to check that $\L(\phi^2) = 2 \phi \L\phi$. One thus obtains
\begin{align*}
    \esp{\abs{\phi(\rho_\infty(t_{k+1})) - \phi(\rho_\infty(t_k))}^2}
        &= 2 \esp{\int_{t_k}^{t_{k+1}} \lrpar{\phi(\rho_\infty(s)) - \phi(\rho_\infty(t_k))}\L\phi(\rho_\infty(s)) ds}\\
        &\leq 2 \int_{t_k}^{t_{k+1}} \esp{\abs{\phi(\rho_\infty(s)) - \phi(\rho_\infty(t_k))}^2}^{1/2} \esp{\abs{\L\phi(\rho_\infty(s))}^2}^{1/2} ds,
\end{align*}
owing to the Cauchy-Schwarz inequality. Since $\phi \in \Theta_{\lim}$, one may use the inequality~\eqref {eq:Lphi_moments}, and $\phi$ is bounded. Thus,~\eqref{eq:increment_phi_rho} gives (with $t = t_k$ and $t' = s$)
\begin{equation*}
    \esp{\abs{\phi(\rho_\infty(s)) - \phi(\rho_\infty(t_k))}^2} \lesssim s-t_k.
\end{equation*}
Using~\eqref {eq:Lphi_moments}, one obtains
\begin{equation} \label{eq:estimate_quadratic_variation_1}
    \esp{\abs{\phi(\rho_\infty(t_{k+1})) - \phi(\rho_\infty(t))}^2} \lesssim \int_{t_k}^{t_{k+1}} (s - t_k)^{1/2} ds \lesssim \Delta^{3/2}.
\end{equation}

On the other hand,~\eqref {eq:Lphi_moments} yields
\begin{equation} \label{eq:estimate_quadratic_variation_2}
    \esp{\abs{\int_{t_k}^{t_{k+1}} \L\phi(\rho_\infty(s)) ds}^2} \lesssim \Delta^2.
\end{equation}

Finally,~\eqref{eq:estimate_quadratic_variation},~\eqref{eq:estimate_quadratic_variation_1} and~\eqref{eq:estimate_quadratic_variation_2} yield
\begin{equation*}
    \esp{\abs{M_\phi(t)}^2} \lesssim \Delta^{1/2} + \Delta \xrightarrow[\Delta \to 0]{} 0.
\end{equation*}

This concludes the proof that $\esp{\abs{M_\phi(t)}^2} = 0$. We deduce that, for all $t \in [0,T]$, almost surely, $M_\phi(t) = 0$. Since $\rho_\infty \in \C^0_T H^{-\varsigma}$, $M_\phi$ is a continuous process. As a consequence, almost surely, for all $t \in [0,T]$, $M_\phi(t) = 0$, which concludes the proof of Proposition~\ref{prop:limit_martingale}.

\end{proof}

\section{Proof of the auxiliary results}\label{sec:proofs}

\subsection{Asymptotic behavior of the stopping time}

The goal of this section is to provide the proof of Proposition~\ref{prop:taue_to_infty}: $\tau^\delta\to\infty$ in probability when $\delta\to 0$.

\begin{proof}[Proof of Proposition~\ref{prop:taue_to_infty}]
Let us first prove the first claim: $\tau_m^\delta\to\infty$ in probability.

Introduce the following sequence of real-valued random variables: for all $i\in\{0,1,\ldots\}$, set
\begin{equation*}
		S_i \doteq \sup_{t \in [i,i+1]} \norm{m(t)}_E.
\end{equation*}
The random variables $S_i$ are almost surely finite, since $\E[S_i^\gamma]<\infty$ for all $i\ge 0$ owing to Assumption~\ref{ass:moments_m_gamma}.

Owing to the condition $\alpha>2/\gamma$, there exists $\alpha'$ such that $2/\gamma<\alpha'<\alpha$. Using Markov inequality and Assumption~\ref{ass:moments_m_gamma}, one obtains
\begin{equation*}
		\sum_{i=1}^\infty \proba{S_i \geq i^{\frac{\alpha'}{2}}} \leq \sum_{i=1}^{\infty} \frac{\esp{S_i^\gamma}}{i^{\frac{\alpha'\gamma}{2}}} = \sup_{i \geq 1} \esp{S_i^\gamma} \sum_{i=1}^{\infty} \frac{1}{i^{\frac{\alpha'\gamma}{2}}} < \infty.
\end{equation*}
Using Borel-Cantelli's lemma, there exists a $\mathbb N$-valued random variable $I_0$, such that almost surely,
\begin{equation*}
		\forall i > I_0, S_i < i^{\frac{\alpha'}{2}}.
\end{equation*}
Define the random variable $Z=\underset{0\le i\le I_0}\sup S_i$. Since $I_0$ is almost surely finite, $Z$ is also an almost surely finite random variable. Observe that almost surely, for all $t\ge 0$, one has
\begin{equation*}
		\norm{m^\delta(t)}_E \leq S_{\floor{t \delta^{-2}}}\leq Z + \floor{t \delta^{-2}}^{\frac{\alpha'}{2}} \leq Z + \lrpar{t \delta^{-2}}^{\frac{\alpha'}{2}}.
\end{equation*}
We are now in position to conclude the proof of the first claim: for all $T\in(0,\infty)$, one obtains
\begin{equation*}
		\proba{\tau^\delta_m < T} = \proba{\sup_{t \in [0,T]} \norm{m^\delta(t)}_E > \delta^{-\alpha}} \leq \proba{Z + T^{\frac{\alpha'}{2}} \delta^{-\alpha'} > \delta^{-\alpha}} \xrightarrow[\delta \to 0]{} 0,
\end{equation*}
since $\alpha'<\alpha$. Thus $\tau_m^\delta\to \infty$ in probability.

It remains to prove the second claim: $\tau_\zeta^\delta\to\infty$ in probability. Let $p\in(\alpha,1)$ be an arbitrary real number. Owing to the Markov inequality and to the continuous embedding $H^{\floor{d/2}+2}_x \subset \C^1_x$, for all $T\in(0,\infty)$ and for all $\delta\in(0,\epsilon_0]$, one has
	\begin{align*}
		\proba{\tau^\delta_\zeta < T}
			&=\proba{\tau^\delta_\zeta < T, \tau^\delta_\zeta < \tau^\delta_m} + \proba{\tau^\delta_\zeta < T, \tau^\delta_\zeta \geq \tau^\delta_m}\\
			&\leq \proba{\delta\sup_{t \in [0,T]} \norm{\zeta\dstopped(t)}_{\C^1_x} \geq 1} + \proba{\tau^\delta_m < T}\\
			&\leq \delta^{2}  \esp{\sup_{t \in [0,T]} \norm{\zeta\dstopped(t)}_{\C^1_x}^2} + \proba{\tau^\delta_m < T}\\
			&\leq \delta^{2(1-p)} \sup_\delta \esp{\delta^{2p} \sup_{t \in [0,T]} \norm{\zeta\dstopped(t)}_{H^{\floor{d/2}+2}_x}^2} + \proba{\tau^\delta_m < T}.
	\end{align*}
Owing to the first claim, it thus remains to prove that
\[
\sup_\delta \esp{\delta^{2p} \sup_{t \in [0,T]} \norm{\zeta\dstopped(t)}_{H^{\floor{d/2}+2}_x}^2}<\infty,
\]
for all $T\in(0,\infty)$.

Let $\beta$ be an arbitrary multi-index of size $|\beta|\leq \floor{d/2}+2$. For all $\delta \in (0,\delta_0]$ and $T\in(0,\infty)$, one has
	\begin{equation} \label{eq:pre_zeta_moment}
		\esp{\sup_{t \in [0,T]} \norm{\frac{\partial^{\abs{\beta}} \zeta\dstopped(t)}{\partial x^\beta}}_{L^2_x}^2} \leq \int_{\T^d} \esp{\sup_{t \in [0,T]} \abs{\frac{\partial^{\abs{\beta}} \zeta\dstopped(t,x)}{\partial x^\beta}}^2}dx.
	\end{equation}
Below an upper bound for the expectation on the right-hand side of~\eqref{eq:pre_zeta_moment} is obtained using properties of a well-chosen martingale. For all $x\in\T^d$ and all multi-indices $\beta$	such that $|\beta|\leq \floor{d/2}+2$, introduce the auxiliary function $\theta_x^\beta$ defined by
	\begin{equation*}
		\theta_x^\beta(\n) = \frac{\partial^{\abs{\beta}} \sigma(\n)}{\partial x^\beta}(x),
	\end{equation*}
for all $n\in E$. Observe that $\theta_x^\beta$ is an element of the set 	$E^*\!\lrpar{\sigma}$, and one may define $\psi_x^\beta \doteq R_0 (\theta_x^\beta-\theta_x^\beta(\bar \sigma))$, which solves of the Poisson equation $-\L_m \psi_x^\beta =  \lrpar{\theta_x^\beta-\theta_x^\beta(\bar \sigma)}$, see Definition~\ref{def:resolvent}. Observe that one has the identities	
	\begin{align*}
		\frac{\partial^{\abs{\beta}} \zeta^\delta(t,x)}{\partial x^\beta} &= \frac{1}{\delta} \int_0^t \frac{\partial^{\abs{\beta}} (\sigma(m^\delta(s))-\bar \sigma)}{\partial x^\beta}(x) ds\\
		&= \frac{1}{\delta} \int_0^t \lrpar{\theta_x^\beta(m^\delta(s)) - \bar \theta_x^\beta} ds\\
		&=-\frac{1}{\delta} \int_0^t \L_m\psi_x^\beta(m^\delta(s)) ds.
	\end{align*}
Finally, for all $x\in\T^d$, $\delta\in(0,\delta_0]$ and $t\ge 0$, set
	\begin{align}
		M^\delta_{\delta^{1+p} \psi_x^\beta}(t)
			&= \delta^{1+p} \psi_x^\beta(m^\delta(t)) - \delta^{1+p} \psi_x^\beta(m(0)) - \frac{1}{\delta^2} \int_0^t \delta^{1+p} \L_m \psi_x^\beta(m^\delta(s))ds \nonumber\\
			&= \delta^{1+p} \psi_x^\beta(m^\delta(t)) - \delta^{1+p} \psi_x^\beta(m(0)) + \delta^p \frac{\partial^{\abs{\beta}} \zeta^\delta(t,x)}{\partial x^\beta} \label{eq:martingale_moment_zeta}
	\end{align}
Then the stopped process $(M^\delta_{\delta^{1+p}})^{\tau^\delta}=M^\delta_{\delta^{1+p}}(\cdot\wedge\tau^{\delta})$ is a martingale. In addition, one obtains the upper bound
\[
\esp{\delta^{2p}\sup_{t \in [0,T]} \abs{\frac{\partial^{\abs{\beta}} \zeta\dstopped(t,x)}{\partial x^\beta}}^2}\leq 2\esp{\sup_{t \in [0,T]} \abs{\delta^{1+p} \psi_x^\beta(m\dstopped(t))}^2}+\esp{\sup_{t \in [0,T]} \abs{M\dstopped_{\delta^{1+p} \psi_x^\beta}(t)}^2}.
\]
On the one hand, owing to estimates~\eqref{eq:estimate_resolvent} and~\eqref{eq:taue_bounds_m}, one obtains
	\begin{equation*}
		\esp{\sup_{t \in [0,T]} \abs{\delta^{1+p} \psi_x^\beta(m\dstopped(t))}^2} \lesssim \delta^{2 + 2p} \lrpar{1 + \delta^{- 2\alpha}} \lesssim 1,
	\end{equation*}
since $\alpha<1$.

On the other hand, Doob's Maximal Inequality yields
	\begin{align*}
		\esp{\sup_{t \in [0,T]} \abs{M\dstopped_{\delta^{1+p} \psi_x^\beta}(t)}^2} &\leq 4 \esp{\abs{M\dstopped_{\delta^{1+p} \psi_x^\beta}(T)}^2}\\
&\leq \frac{4}{\delta^2} \esp{(\delta^{1+p})^2 \int_0^{T \wedge \tau^\delta} \lrpar{\L_m\lrpar{\lrpar{\psi_x^\beta}^2} - 2 \psi_x^\beta \L_m\psi_x^\beta}(m^\delta(s))ds}\\
&\lesssim_T \delta^{2p} \lrpar{1 + \delta^{- 2\alpha}}\\
&\lesssim 1,
	\end{align*}
since $p\in(\alpha,1)$, using Assumption~\ref{ass:regular_Lm} and the estimates~\eqref{eq:estimate_resolvent} and~\eqref{eq:taue_bounds_m}.

Since all the bounds obtained above are uniform with respect to $x\in\T^d$, gathering the estimates one obtains
\eqref{eq:martingale_moment_zeta} yields
	\begin{equation*}
\int_{x\in\T^d}\esp{\delta^{2p} \sup_{t \in [0,T]} \abs{\frac{\partial^{\abs{\beta}} \zeta\dstopped(t,x)}{\partial x^\beta}}^2}dx \lesssim 1.
	\end{equation*}
The arguments explained above then yield the result: for all $T\in(0,\infty)$, 	
\[
\proba{\tau^\delta_\zeta < T}\lesssim \delta^{2(1-p)}+\proba{\tau^\delta_m < T}\underset{\delta\to 0}\to 0,
\]	
therefore one has $\tau^{\delta}_{\zeta}\to\infty$ in probability.
	
Since $\tau^\delta=\tau_m^\delta\wedge\tau_\zeta^\delta$, the proof of Proposition~\ref{prop:taue_to_infty} is completed.
\end{proof}

\subsection{A priori estimate of the solution in $\fSet$}

The goal of this section is to provide the proof of Proposition~\ref{prop:L2_bound}, which gives an almost sure a priori estimate for $f^{\epsilon,\delta,\tau^\delta}$ in the space $C_T^0\fSet$ and for $Lf^{\epsilon,\delta,\tau^\delta}$ in the space $L^2_{t\wedge \tau^\delta}\fSet$, in terms of the norm of the initial condition $\norm{f^{\epsilon,\delta}_0}_\fSet$.

\begin{proof}[Proof of Proposition~\ref{prop:L2_bound}]
For all $\delta\in(0,\delta_0]$ and all $t\ge 0$, set
\begin{equation*}
		\eta^\delta(t,\cdot) = \int_0^t \sigma(m^\delta(s))(\cdot) ds = \delta \zeta^\delta(t,\cdot) + t \bar \sigma(\cdot)\in \C^{\floor{d/2}+2}_x,
	\end{equation*}
where $\zeta^\delta$ is defined by\eqref{eq:def_zeta_epsilon}. Note that, owing to~\eqref{eq:taue_bounds_zeta}, for all $\delta\in(0,\delta_0]$, and $i=0,1$, one has
	\begin{equation} \label{eq:taue_bounds_eta}
		\underset{t\in[0,\tau^\delta\wedge T]}\sup~\norm{\eta^\delta(t)}_{\C^i_x} \leq 1 + T \norm{\bar \sigma}_{\C^i_x}.
	\end{equation}

Let us now introduce the following family of weight functions $\M^\delta$ indexed by $\delta\in(0,\delta_0]$: for all $t\geq 0$, $x\in\T^d$ and $v\in V$, set
	\begin{equation*}
		\M^\delta(t,x,v) = \exp\lrpar{-2 \eta^\delta(t,x)} \M(v),
	\end{equation*}
The associated weighted $L^2$ norm is defined by
	\begin{equation*}
		\norm{f}_{L^2(\M^\delta(t)^{-1})}^2 \doteq \iint\frac{\abs{f(x,v)}^2}{\M^\delta(t,x,v)} dx d\mu(v).
	\end{equation*}
	Note that for all $h \in \fSet$ and all $t \in [0,\tau^\delta\wedge T]$, the inequality~\eqref{eq:taue_bounds_eta} yields
	\begin{equation}\label{eq:equivnorms}
		\norm{h}_\fSet^2 \leq \norm{h}_{L^2(\M^\delta(t)^{-1})}^2 e^{2 + 2 T \norm{\bar \sigma}_{\C^0_x}}.
	\end{equation}
It is thus sufficient to prove estimates in the weight norm $\norm{\cdot}_{L^2(\M^\delta(t)^{-1})}$, to retrieve the a priori estimates in the space $\fSet$, when the condition $t\le \tau^\delta\wedge T$ is satisfied.

	For all $t\ge 0$, one has
	\begin{align*}
		\frac{1}{2} \partial_t \norm{f^{\epsilon,\delta}(t)}_{L^2(\M^\delta(t)^{-1})}^2
			&= \iint \frac{f^{\epsilon,\delta}(t,x,v)}{\M^\delta(t,x,v)} \partial_t f^{\epsilon,\delta}(t,x,v) dx d\mu(v)\\
			&\phantom{=}\ - \iint \frac{\abs{f^{\epsilon,\delta}(t,x,v)}^2}{2\abs{\M^\delta(t,x,v)}^2} \partial_t \M^\delta(t,x,v) dx d\mu(v)\\
			&= \mathcal A_{\epsilon,\delta}(t) + \mathcal B_{\epsilon,\delta}(t) + \mathcal C_{\epsilon,\delta}(t)
	\end{align*}
	with
	\begin{gather*}
		\mathcal A_{\epsilon,\delta}(t) = \frac{1}{\epsilon^2} \iint \frac{f^{\epsilon,\delta}(t,x,v)}{\M^\delta(t,x,v)} Lf^{\epsilon,\delta}(t,x,v) dx d\mu(v)\\
		\mathcal B_{\epsilon,\delta}(t) = - \frac{1}{\epsilon} \iint \frac{f^{\epsilon,\delta}(t,x,v)}{\M^\delta(t,x,v)} (a(v) + \epsilon b(v)) \cdot \nabla_x f^{\epsilon,\delta}(t,x,v) dx d\mu(v)\\
		\mathcal C_{\epsilon,\delta}(t) = - \iint \frac{\abs{f^{\epsilon,\delta}(t,x,v)}^2}{\M^\delta(t,x,v)} \lrpar{\sigma(m^\delta) + \frac{\partial_t \M^\delta}{2\M^\delta}}(t,x,v) dx d\mu(v).
	\end{gather*}
Note that the third term vanishes: $\mathcal C_{\epsilon,\delta}(t) = 0$ for all $t \ge 0$. Indeed, the definition of the weight function $\M^\delta$ yields the identity $\sigma(m^\delta) + \frac{\partial_t \M^\delta}{2\M^\delta} = 0$.

Using the identity $f^{\epsilon,\delta} = \rho^{\epsilon,\delta} \M - Lf^{\epsilon,\delta}$ and the property $\int_{V} Lf^{\epsilon,\delta}(t,x,v) d\mu(v) = 0$, the first term $\mathcal A_{\epsilon,\delta}(t)$ is written as follows: for all $t\ge 0$,
	\begin{align*}
		\mathcal A_{\epsilon,\delta}(t)
			&= \frac{1}{\epsilon^2} \iint \frac{f^{\epsilon,\delta}(t,x,v)}{\M^\delta(t,x,v)} Lf^{\epsilon,\delta}(t,x,v) dx d\mu(v)\\
			&= \frac{1}{\epsilon^2} \int_{\T^d} e^{2 \eta^\delta(t,x)} \rho^{\epsilon,\delta}(t,x) \int_{V} Lf^{\epsilon,\delta}(t,x,v) d\mu(v) dx - \frac{1}{\epsilon^2} \iint \frac{\abs{Lf^{\epsilon,\delta}(t,x,v)}^2}{\M^\delta(t,x,v)} d\mu(v) dx\\
			&= - \frac{1}{\epsilon^2} \norm{L f^{\epsilon,\delta}(t)}_{L^2(\M^\delta(t)^{-1})}^2.
	\end{align*}

The treatment of the second term $\mathcal B_{\epsilon,\delta}(t)$ requires technical computations. Using an integration by parts arguments, for all $t\ge 0$ one obtains
	\begin{align*}
		\mathcal B_{\epsilon,\delta}(t)
			&= - \frac{1}{\epsilon} \iint (a(v) + \epsilon b(v)) \cdot \frac{f^{\epsilon,\delta}(t,x,v) \nabla_x f^{\epsilon,\delta}(t,x,v)}{\M^\delta(t,x,v)} dx d\mu(v)\\
			&= - \frac{1}{\epsilon} \iint (a(v) + \epsilon b(v)) \cdot \frac{\frac{1}{2} \abs{f^{\epsilon,\delta}(t,x,v)}^2 \nabla_x \M^\delta(t,x,v)}{\abs{\M^\delta(t,x,v)}^2} dx d\mu(v)\\
			&=\frac{1}{\epsilon} \iint (a(v) + \epsilon b(v)) \cdot \nabla_x \eta^\delta(t,x) \frac{\abs{f^{\epsilon,\delta}(t,x,v)}^2}{\M^\delta(t,x,v)} d\mu(v) dx.
	\end{align*}
Using again the identity $f^{\epsilon,\delta} = \rho^{\epsilon,\delta} \M - Lf^{\epsilon,\delta}$ then gives, for all $t\ge 0$,
	\begin{align*}
		\mathcal B_{\epsilon,\delta}(t)
			&= \frac{1}{\epsilon} \int_{\T^d} e^{2 \eta^\delta(t,x)} \abs{\rho^{\epsilon,\delta}(t,x)}^2 \nabla_x \eta^\delta(t,x) \cdot \lrpar{\int_V (a(v) + \epsilon b(v)) \M(v) d\mu(v)} dx\\
			&\phantom{=}\ + \frac{1}{\epsilon} \iint (a(v) + \epsilon b(v)) \cdot \nabla_x \eta^\delta(t,x) \frac{\abs{Lf^{\epsilon,\delta}(t,x,v)}^2}{\M^\delta(t,x,v)} d\mu(v) dx\\
			&\phantom{=}\ - \frac{2}{\epsilon} \iint (a(v) + \epsilon b(v)) \cdot \nabla_x \eta^\delta(t,x) \rho^{\epsilon,\delta}(t,x) Lf^{\epsilon,\delta}(t,x,v) \frac{\M(v)}{\M^\delta(t,x,v)} d\mu(v) dx\\
			&= \mathcal B_{\epsilon,\delta}^1(t) + \mathcal B_{\epsilon,\delta}^2(t) + \mathcal B_{\epsilon,\delta}^3(t).
	\end{align*}
Let us now treat successively the three terms appearing on the right-hand side above.
	\begin{itemize}
		\item Owing to Assumption \ref{ass:coefficients} and to the definition of $J$, one has $\int_V (a(v)+\epsilon b(v)) \M(v) d\mu(v) = \epsilon J$. Therefore
		\begin{align*}
		    \abs{\mathcal B_{\epsilon,\delta}^1(t)}
		        &= \abs{\int_{\T^d} e^{2 \eta^\delta(t,x)} \abs{\rho^{\epsilon,\delta}(t,x)}^2 J \cdot \nabla_x \eta^\delta(t,x) dx}\\
		        &\leq \norm{b}_{L^\infty} \norm{\nabla_x \eta^\delta(t)}_{C_x} \int_{\T^d} e^{2 \eta^\delta(t,x)} \abs{\rho^{\epsilon,\delta}(t,x)}^2 dx.
		\end{align*}
		Owing to the Cauchy-Schwarz inequality, one obtains
		\begin{align}
			\int_{\T^d} e^{2 \eta^\delta(t,x)} \abs{\rho^{\epsilon,\delta}(t,x)}^2 dx
				&\leq \int_{\T^d} e^{2 \eta^\delta(t,x)} \int_V \frac{\abs{f^{\epsilon,\delta}(t,x,v)}^2}{\M^\delta(t,x,v)} d\mu(v) \int_V \M^\delta(t,x,v) d\mu(v) dx \nonumber\\
				&= \norm{f^{\epsilon,\delta}(t)}_{L^2(\M^\delta(t)^{-1})}^2 \label{eq:rho_bound},
		\end{align}
		since $\int_V \M^\delta(t,x,v) d\mu(v)= e^{-2 \eta^\delta(t,x)}$ for all $t\ge 0$ and all $x\in\T^d$. As a consequence, using the inequality~\eqref{eq:taue_bounds_eta}, for all $t \in [0,\tau^\delta\wedge T]$, one has
		\begin{equation*}
			\abs{\mathcal B_{\epsilon,\delta}^1(t)} \leq \norm{b}_{L^\infty} (1 + T \norm{\bar \sigma}_{\C^1_x}) \norm{f^{\epsilon,\delta}(t)}_{L^2(\M^\delta(t)^{-1})}^2.
		\end{equation*}

		\item Using the condition $\epsilon\le\epsilon_0$ with $\epsilon_0$ satisfying~\eqref{eq:def_epsilon_0}, and the inequality~\eqref{eq:taue_bounds_eta}, for $t \in [0, \tau^\delta\wedge T]$ one has
		\begin{equation*}
			\abs{\mathcal B_{\epsilon,\delta}^2(t)} \leq \frac{1}{4\epsilon^2} \norm{L f^{\epsilon,\delta}(t)}_{L^2(\M^\delta(t)^{-1})}^2.
		\end{equation*}

		\item Using Young's inequality, then using the inequalities~\eqref{eq:taue_bounds_eta} and~\eqref{eq:rho_bound}, one obtains, for all $t\in[0,\tau^\delta\wedge T]$,
		\begin{align*}
			\abs{\mathcal B_{\epsilon,\delta}^3(t)} &\leq 4 (\norm{a}_{L^\infty} + \norm{b}_{L^\infty})^2 \norm{\nabla_x \eta^\delta(t)}_{C_x}^2 \int_{\T^d} e^{2 \eta^\delta(t,x)} \abs{\rho^{\epsilon,\delta}(t,x)}^2 dx + \frac{1}{4\epsilon^2} \norm{L f^{\epsilon,\delta}(t)}_{L^2(\M^\delta(t)^{-1})}^2\\
&\leq 4 (\norm{a}_{L^\infty} + \norm{b}_{L^\infty})^2 (1 + T \norm{\bar \sigma}_{\C^1_x})^2 \norm{f^{\epsilon,\delta}(t)}_{L^2(\M^\delta(t)^{-1})}^2 + \frac{1}{4\epsilon^2} \norm{L f^{\epsilon,\delta}(t)}_{L^2(\M^\delta(t)^{-1})}^2.	
		\end{align*}
	\end{itemize}

Gathering the estimates, one obtains the following inequalities: for all $t\in[0,\tau^\delta\wedge T]$,
\[
\abs{\mathcal B_{\epsilon,\delta}(t)}\leq C_0(T)\norm{f^{\epsilon,\delta}(t)}_{L^2(\M^\delta(t)^{-1})}^2+\frac{1}{2\epsilon^2} \norm{L f^{\epsilon,\delta}(t)}_{L^2(\M^\delta(t)^{-1})}^2
\]
and
\[
\frac{1}{2} \partial_t \norm{f^{\epsilon,\delta}(t)}_{L^2(\M^\delta(t)^{-1})}^2+\frac{1}{2\epsilon^2} \norm{L f^{\epsilon,\delta}(t)}_{L^2(\M^\delta(t)^{-1})}^2\leq C_0(T)\norm{f^{\epsilon,\delta}(t)}_{L^2(\M^\delta(t)^{-1})}^2.
\]
where $C_0(T)\in(0,\infty)$ is a deterministic real number, and does not depend on $\epsilon,\delta$. Applying Gronwall's inequality, for all $t\in[0,\tau^{\delta}\wedge T]$, one gets
	\begin{equation*}
		\norm{f^{\epsilon,\delta}(t)}_{L^2(\M^\delta(t)^{-1})}^2 + \frac{1}{2\epsilon^2}\int_0^t  \norm{L f^{\epsilon,\delta}(s)}_{L^2(\M^\delta(t)^{-1})}^2 ds \leq C(T)\norm{f^{\epsilon,\delta}_0}_\fSet^2,
	\end{equation*}
where $C(T)\in(0,\infty)$ is a deterministic real number, and does not depend on $\epsilon,\delta$. Finally, using~\eqref{eq:equivnorms}, one obtains the a priori estimate~\eqref{eq:L2_bound}, which concludes the proof of Proposition~\ref{prop:L2_bound}.

\end{proof}

\subsection{Construction of the perturbed test function $\varphi^{\epsilon,\delta}$}

The objective of this section is to prove Proposition~\ref{prop:corrector}: more precisely, given a test function $\varphi\in\Theta_{\lim}$ (which only depends on $\rho=\lrangle f$), we construct the four correctors $\varphi_{1,0}$, $\varphi_{2,0}$, $\varphi_{0,2}$ and $\varphi_{1,2}$, such that the perturbed test function $\varphi^{\epsilon,\delta}$ defined by~\eqref{eq:def_phi_epsilon} satisfies the error estimates~\eqref{eq:control_perturbed_generator} and~\eqref{eq:control_error_phi}, and appropriate upper bounds.

Below, we first state auxiliary results concerning solutions of Poisson equations. The expression of $\L^{\epsilon,\delta}\varphi^{\epsilon,\delta}$ is then expanded in powers of $\epsilon$ and $\delta$, and the resulting equality yields a family of equations to be satisfied by the correctors in order to satisfy~\eqref{eq:control_perturbed_generator}. The correctors are finally constructed successively as solutions of appropriate Poisson equations. Eventually, it only remains to check the required regularity properties and upper bounds, this step follows from straightforward computations.

\subsubsection{Auxiliary results on Poisson equations}

As will be clear below, the construction of the correctors requires to solve Poisson equations of the type $-\L_2\psi(f,\n)=\vartheta(f,\n)$ (where $\n$ is considered as a fixed parameter) and $-\L_m\psi(f,\n)=\vartheta(f,\n)$ (where $f$ is considered as a fixed parameter). We describe below the corresponding centering conditions which are needed for the solvability of those Poisson equations, and give the expressions of the solutions.

To solve the first class of Poisson equations, let us introduce the process $\seq{g_f(t)}{t \in \R^+}$, associated with the infinitesimal generator $\L_2$, with the initial condition $g_f(0) = f$: for all $t\ge 0$, one has
\begin{equation*}
	g_f(t) = \rho \M + e^{-t} \lrpar{f - \rho \M},
\end{equation*}
where $\rho=\lrangle f=\lrangle g_f(t)$ for all $t\ge 0$.

The solvability of the first class of Poisson equations $-\L_2\psi=\vartheta$ is ensured when the following centering condition is satisfied: for all $\rho\in L_x^2$ and all $\n\in E$,
\begin{equation} \label{eq:centering_condition_L2}
	\vartheta(\rho \M,\n) = 0.
\end{equation}
If~\eqref{eq:centering_condition_L2} is satisfied, then the function $\psi$ defined by
\begin{equation} \label{eq:Poisson_solution_L2}
	\psi(f,\n) = \int_0^\infty \vartheta(g_f(t),\n) dt,
\end{equation}
for all $f\in\fSet$ and $\n\in E$, is a solution of the Poisson equation $-\L_2\psi=\vartheta$. It is the unique solution such that $\psi(\rho\M,\n)=0$ for all $\rho\in L_x^2$ and $\n\in E$.

The solvability of the second class of Poisson equations $-\L_m\psi=\vartheta$ is ensured when the following centering condition is satisfied: for all $f\in\fSet$,
\begin{equation} \label{eq:centering_condition_Lm}
	\int_E \vartheta(f,\n) d\nu(\n) = 0.
\end{equation}
If~\eqref{eq:centering_condition_Lm} is satisfied, then the function $\psi$ defined by
\begin{equation} \label{eq:Poisson_solution_Lm}
	\psi(f,\n) = \int_0^\infty \esp{\vartheta(f,m_\ell(t))} dt.
\end{equation}
for all $f\in\fSet$ and $\n\in E$, is a solution of the Poisson equation $-\L_m\psi=\vartheta$. It is the unique solution such that $\int\psi(f,\n) d\nu(\n) = 0$ for all $f\in\fSet$.

In the sequel, the following class of functions $\vartheta$ is considered. For all $h, k \in L_x^2$, define $\theta_{h,k}(\n) = \lrpar{h \sigma(\n),k}_{L_x^2}$ for all $\n\in E$. Observe that $\theta_{h,k} \in E^*\!\lrpar{\sigma}$ (see Definition~\ref{def:resolvent}). For all $f\in\fSet$ and $\n\in E$, set $\vartheta_{h,k}(f,\n)=\theta_{h,k}(n)-\bar\theta_{h,k}$,  with $\bar\theta_{h,k}=\int \theta_{h,k}(\n)d\nu(\n)=\lrpar{h\bar\sigma,k}_{L_x^2}$. For such functions $\vartheta_{h,k}$, the solution of the Poisson equation $-\L_m\psi_{h,k}=\vartheta_{h,k}=\theta_{h,k}-\bar\theta_{h,k}$ is given by
\[
\psi_{h,k}=R_0\lrpar{\theta_{h,k}(\ell) - \bar \theta_{h,k}},
\]
where the resolvent operator $R_0$ is introduced in Definition~\ref{def:resolvent}. Using the estimate~\eqref{eq:estimate_resolvent} and the to Riesz representation Theorem, for all $\n\in E$, there exists a bounded linear operator $R_0(\n) : L_x^2 \to L_x^2$ such that, for all $\n \in E$ and $h,k \in L_x^2$, $\psi_{h,k}(\n) = \lrpar{R_0(\n) h,k}_{L^2}$. Let us state some useful properties of the operators $R_0(\n)$.
	\begin{enumerate}
		\item[$(i)$] For all $\n \in E$, $R_0(\n)$ is self-adjoint.

		\item[$(ii)$] For all $\n \in E$ and $h\in L_x^2$, one has
		\begin{equation*}
			\norm{R_0(\ell)h}_{L^2} \lesssim \norm{h}_{L^2} \lrpar{1 + \norm{\ell}_E}.
		\end{equation*}

		\item[$(iii)$] For all $\n\in E$ and all $h\in H_x^1$, one has $R_0(\n)h\in H_x^1$, and
		\begin{equation*}
			\norm{R_0(\n)h}_{H^1_x} \lesssim \norm{h}_{H^1_x} \lrpar{1 + \norm{\n}_E}.
		\end{equation*}
	\end{enumerate}

The proof of Claim $(i)$ is straightforward: for all $h,k\in L_x^2$, $\theta_{h,k}=\theta_{k,h}$, thus $\psi_{h,k} = \psi_{k,h}$, which gives for all $\n\in E$
		\begin{equation*}
			\lrpar{R_0(\n) h, k}_{L^2} = \lrpar{h, R_0(\n) k}_{L^2}.
		\end{equation*}
Claim $(ii)$ follows from the estimate~\eqref{eq:estimate_resolvent}: one has $\abs{\psi_{h,k}(\ell)} \lesssim \norm{h}_{L^2} \norm{k}_{L^2} \lrpar{1 + \norm{\ell}_E}$, since ${\rm Lip}(\theta_{h,k})\leq {\rm Lip}(\sigma)\norm{h}_{L^2} \norm{k}_{L^2}$.

Finally, Claim $(iii)$ is obtained as follows: since $E\in \C^1_x$, one has $\norm{\theta_{h,k}}_{\Lip(E)} \leq \norm{h}_{H^1_x} \norm{k}_{H^{-1}_x}$ if $h \in H^1_x$ and $k \in L^2_x$. Using the estimate~\eqref{eq:estimate_resolvent} then gives $\abs{\psi_{h,k}(\ell)} \lesssim \norm{h}_{H^1_x} \norm{k}_{H^{-1}_x} \lrpar{1 + \norm{\ell}_E}$.

\subsubsection{Multiscale expansion and family of Poisson equations}

Let $\varphi^{\epsilon,\delta}$ be of the form~\eqref{eq:def_phi_epsilon}, where the correctors are not known at this stage. Then $\L^{\epsilon,\delta}\varphi^{\epsilon,\delta}$ is expressed as follows:
\begin{align*}
	\L^{\epsilon,\delta} \phi^{\epsilon,\delta}
		&= \epsilon^{-2} \L_2 \phi + \delta^{-2} \L_m \phi\\
		&\phantom{=}\ + \epsilon^{-1} \lrpar{\L_1 \phi + \L_2 \phi_{1,0}} + \epsilon \delta^{-2} \L_m \phi_{1,0}\\
		&\phantom{=}\ + \lrpar{\L_0 \phi + \L_1 \phi_{1,0} + \L_2 \phi_{2,0} + \L_m \phi_{0,2}} + \epsilon^{-2} \delta^2 \L_2 \phi_{0,2} + \epsilon^2 \delta^{-2} \L_m \phi_{2,0}\\
		&\phantom{=}\ + \epsilon \lrpar{\L_0 \phi_{1,0} + \L_1 \phi_{2,0} + \L_m \phi_{1,2}} + \epsilon^{-1} \delta^2 \lrpar{\L_1 \phi_{0,2} + \L_2 \phi_{1,2}}\\
		&\phantom{=}\ + \epsilon^2 \L_0 \phi_{2,0} + \delta^2 \L_0 \phi_{0,2} + \delta^2 \L_1 \phi_{1,2}\\
		&\phantom{=}\ + \epsilon \delta^2 \L_0 \phi_{1,2},
\end{align*}
where each line on the right-hand side corresponds to expressions of degree $-2,\ldots,3$ in terms of the variables $\epsilon,\delta$. The goal is to construct the correctors such that $\L^{\epsilon,\delta}\varphi^{\epsilon,\delta}-\L\varphi$ goes to $0$ when $(\epsilon,\delta) \to (0,0)$, more precisely such that~\eqref{eq:control_perturbed_generator} holds. The following family of conditions provide sufficient conditions on the correctors to satisfy~\eqref{eq:control_perturbed_generator}: on the one hand, the correctors solve the following system of equations,
\begin{align}
	&\L_2 \phi = \L_m \phi = 0, \label{eq:aim_corrector-2}\\
	&\L_1 \phi + \L_2 \phi_{1,0} = \L_m \phi_{1,0} = 0, \label{eq:aim_corrector-1}\\
	&\L_0 \phi + \L_1 \phi_{1,0} + \L_2 \phi_{2,0} + \L_m \phi_{0,2} = \L \phi, \label{eq:aim_corrector0}\\
	&\L_2 \phi_{0,2} = \L_m \phi_{2,0} = 0, \label{eq:aim_corrector0bis}\\
	&\L_1 \phi_{0,2} + \L_2 \phi_{1,2} = 0, \label{eq:aim_corrector1}
\end{align}
on the other hand, the following estimates are satisfied,
\begin{gather}
	\abs{\L_0 \phi_{1,0}(f,\n)} + \abs{\L_1 \phi_{2,0}(f,\n)} + \abs{\L_m \phi_{1,2}(f,\n)} + \abs{\L_0 \phi_{2,0}(f,\n)} \lesssim (1 + \norm{f}_\fSet^3) (1 + \norm{\n}_E), \label{eq:aim_corrector23_epsilon}\\
	\abs{\L_0 \phi_{0,2}(f,\ell)} + \abs{\L_1 \phi_{1,2}(f,\n)} + \abs{\L_0 \phi_{1,2}(f,\n)} \lesssim (1 + \norm{f}_\fSet^3) (1 + \norm{\n}_E^2), \label{eq:aim_corrector23_delta}
\end{gather}
for all $f\in \fSet$ and $\n\in E$. The estimates~\eqref{eq:control_phi_10},~\eqref{eq:control_phi_20},~\eqref{eq:control_phi_20} and~\eqref{eq:control_phi_12} are byproducts of the constructions of the correctors below.

Observe that the last equation in the system, Equation~\eqref{eq:aim_corrector1}, would not appear if $\epsilon=\delta$, or if a constraint of the type $\epsilon^{-1}\delta^2\to 0$ is satisfied. In the expression of $\L^{\epsilon,\delta}\varphi^{\epsilon,\delta}$ above, the condition~\eqref{eq:aim_corrector1} corresponds to a contribution of the term of degree $1$. One of the novelties of this work is to consider the general case, hence the need to construct the corrector $\varphi_{1,2}$. On the contrary, as will be clear below, the condition~\eqref{eq:aim_corrector0bis} on the correctors $\varphi_{0,2}$ and $\varphi_{2,0}$ is simpler to treat, it only means that $\varphi_{0,2}(f,\n)=\varphi_{0,2}(\rho\M,\n)$ and $\varphi_{2,0}(f,\n)=\varphi_{2,0}(f)$. In the case $\epsilon=\delta$, this only consists in writing the corrector $\varphi_2$ of order $2$ as a sum of two terms $\varphi_{2,0}+\varphi_{0,2}$.

Recall that the test function $\varphi$ belongs to the class of test functions $\Theta_{\lim}$ (see Definition~\ref{def:Theta-lim}): $\varphi(f,\n)=\chi(\lrpar{\rho,\xi}_{L^2_x})$. In the sequel, to simplify the expressions, we use the notation $\chi_\rho = \chi(\lrpar{\rho,\xi}_{L^2_x})$, $\chi'_\rho = \chi'(\lrpar{\rho,\xi}_{L^2_x})$, and similar notation for the higher order derivatives $\chi_\rho''$ and $\chi_\rho'''$.

Let us explain how the rest of this section proceeds. We first check that~\eqref{eq:aim_corrector-2} holds when $\varphi$ is a function in the class $\Theta_{\lim}$. Then, we successively construct the correctors $\varphi_{1,0}$, $\varphi_{2,0}$ and $\varphi_{0,2}$, and $\varphi_{1,2}$, as solutions of Poisson equations, using the tools above. We finally check that~\eqref{eq:aim_corrector23_epsilon} and~\eqref{eq:aim_corrector23_delta} holds. The proof of Proposition~\ref{prop:corrector} is concluded when all those steps are completed.

\subsubsection{Verification of the condition~\eqref{eq:aim_corrector-2}}

The function $\varphi$ is in the class $\Theta_{\lim}$: as a consequence $\varphi(f,\n)=\varphi(\rho\M)$ does not depend on $\n$, and only depends on $f$ through $\rho=\lrangle f$. It is then straightforward to check that $\L_m\varphi(f,\n)=\L_2\varphi(f,\n)=0$ for all $f\in\fSet$ and $\n\in E$. Therefore~\eqref{eq:aim_corrector-2} holds.

\subsubsection{Construction of the corrector $\varphi_{1,0}$ using the condition~\eqref{eq:aim_corrector-1}}

First, using~\eqref{eq:Poisson_solution_Lm}, the condition $\L_m\varphi_{1,0}=0$ implies that $\varphi_{1,0}$ is independent of $\n$: $\varphi_{1,0}(f,\n)=\varphi(f)$ for all $f\in\fSet,\n\in E$.

Second, $\varphi_{1,0}$ is the solution of the Poisson equation
\[
-\L_2\varphi_{1,0}(f)=\L_1\varphi(f)=-\chi'_\rho \lrpar{\lrangle{Af},\xi}_{L^2_x}.
\]
The centering condition~\eqref{eq:centering_condition_L2} is satisfied for $\vartheta=\L_1\varphi$, indeed $\lrangle{A\rho\M}=0$ for all $\rho\in L_x^2$. Using the expression~\eqref{eq:Poisson_solution_L2}, one thus defines the corrector $\varphi_{1,0}$ as follows: for all $f\in\fSet$ and $\n\in E$, set
\[
\varphi_{1,0}(f,\n)=\int_0^\infty \L_1 \phi(g_f(t)) dt.
\]
Recall that $\lrangle{g_f(t)} = \rho$ and $\lrangle{Ag_f(t)} = e^{-t} \lrangle{Af}$ for all $t\ge 0$, thus one has, for all $f\in\fSet$ and $\n\in E$,
\begin{align}
	\phi_{1,0}(f,\n)
		&= - \int_0^\infty e^{-t} \chi'_\rho \lrpar{\lrangle{Af},\xi}_{L^2_x} dt = - \chi'_\rho \lrpar{\lrangle{Af},\xi}_{L^2_x} \nonumber\\
		&= \chi'_\rho \lrpar{f,A\xi}_{L^2}, \label{eq:def_phi_10}
\end{align}
using an integration by parts argument in the last equality.

Using the expression~\eqref{eq:def_phi_10}, it is then straightforward to check that the estimate~\eqref{eq:control_phi_10} is satisfied. This concludes the construction of the corrector $\phi_{1,0}$.

\subsubsection{Construction of the correctors $\varphi_{2,0}$ and $\varphi_{0,2}$ using the conditions~\eqref{eq:aim_corrector0}--\eqref{eq:aim_corrector0bis}}

It is required to combine the two conditions~\eqref{eq:aim_corrector0} and~\eqref{eq:aim_corrector0bis} in order to build the correctors $\varphi_{2,0}$ and $\varphi_{0,2}$.

First, using~\eqref{eq:Poisson_solution_Lm}, the condition $\L_m\varphi_{2,0}=0$ implies that $\varphi_{2,0}$ is independent of $\n$: $\varphi_{2,0}(f,\n)=\varphi(f)$ for all $f\in\fSet,\n\in E$. Similarly, using~\eqref{eq:Poisson_solution_L2}, the condition $\L_2\varphi_{0,2}=0$ implies that $\varphi_{0,2}(f,\n)=\varphi_{0,2}(\rho\M,\n)$, with $\rho=\lrangle f$, for all $f\in\fSet,n\in E$.

Second, observe that one has
\begin{align*}
	\L_0 \phi(f,\n) + \L_1 \phi_{1,0}(f,\n)
		&= - \chi'_\rho \lrpar{\sigma(\n) \rho,\xi}_{L^2_x} - \chi'_\rho \lrpar{\lrangle{Bf},\xi}_{L^2_x}\\
		&\phantom{=}\ - \chi''_\rho \lrpar{\lrangle{Af},\xi}_{L^2_x} \lrpar{f,A\xi}_{L^2} - \chi'_\rho \lrpar{Af,A\xi}_{L^2}\\
		&= \vartheta_{0,2}(\rho\M,\n) + \vartheta_{2,0}(f),
\end{align*}
for all $f\in\fSet$ and $\n\in E$, where the auxiliary functions $\vartheta_{0,2}$ and $\vartheta_{2,0}$ are defined as follows:
\begin{align*}
    \vartheta_{0,2}(\rho\M,\n) &= - \chi'_\rho \lrpar{\sigma(\n) \rho,\xi}_{L^2_x}\\
    \vartheta_{2,0}(f)
        &= \chi'_\rho \lrpar{f,B\xi}_{L^2} + \chi''_\rho \lrpar{\lrangle{Af},\xi}_{L^2_x}^2 + \chi'_\rho \lrpar{f,A^2\xi}_{L^2}\\
        &= \chi''_\rho \lrpar{\lrangle{Af},\xi}_{L^2_x}^2 + \chi'_\rho \lrpar{f,A^2\xi + B\xi}_{L^2},
\end{align*}
Note that using~\eqref{eq:Lrho-theta}, for all $\rho\in L_x^2$, one has
\begin{align*}
\L\phi(\rho)&= \chi'_\rho \lrpar{\rho,\div_x(K\nabla_x \xi)}_{L^2_x} + \chi'_\rho \lrpar{\rho,J \cdot \nabla_x \xi}_{L^2_x} - \chi'_\rho \lrpar{\rho,\bar \sigma \xi}_{L^2_x}\\
&=\chi'_\rho \lrpar{\rho,K:\nabla_x^2\xi}_{L^2_x} + \chi'_\rho \lrpar{\rho,J \cdot \nabla_x \xi}_{L^2_x} - \chi'_\rho \lrpar{\bar \sigma \rho,\xi}_{L^2_x}\\
&= \chi'_\rho \lrpar{\rho\M,A^2\xi + B\xi}_{L^2} - \chi'_\rho \lrpar{\bar \sigma \rho,\xi}_{L^2_x}\\
&= \vartheta_{2,0}(\rho\M) + \int_E \vartheta_{0,2}(\rho\M,\cdot) d\nu\\
&=\vartheta_{2,0}(\rho\M) + \int_E \vartheta_{0,2}(\rho\M,\cdot) d\nu.
\end{align*}
Indeed, $\lrangle{A\rho\M}=0$ for all $\rho\in L_x^2$.

As a consequence, the correctors $\varphi_{2,0}$ and $\varphi_{0,2}$ are constructed as the solutions of the Poisson equations
\begin{gather}
	-\L_2 \phi_{2,0}(f) = \lrpar{\vartheta_{2,0}(f) - \vartheta_{2,0}(\rho\M)}, \label{eq:Poisson_def_phi_20}\\
	-\L_m \phi_{0,2}(\rho \M,\n) =  \lrpar{\vartheta_{0,2}(\rho \M, \n) - \int_E \vartheta_{0,2}(\rho\M,\cdot) d\nu}. \label{eq:Poisson_def_phi_02}
\end{gather}
Indeed, if~\eqref{eq:Poisson_def_phi_20} and~\eqref{eq:Poisson_def_phi_02} are satisfied, then the arguments above show that the conditions~\eqref{eq:aim_corrector0} and~\eqref{eq:aim_corrector0bis} hold. It remains to solve the two Poisson equations, and to check that the estimates~\eqref{eq:control_phi_20} and~\eqref{eq:control_phi_02} are satisfied.

On the one hand, the centering condition~\eqref{eq:centering_condition_L2} is satisfied, and~\eqref{eq:Poisson_solution_L2} gives the following definition for $\phi_{2,0}$: for all $f\in \fSet$, set
\begin{align}
	\phi_{2,0}(f)
		&= \int_0^\infty \lrpar{\vartheta_{2,0}(g_f(t)) - \vartheta_{2,0}(\rho\M)} dt \nonumber\\
\end{align}		
Using the properties $\lrangle{g_f(t)} = \rho$, $\lrangle{A g_f(t)} = e^{-t} \lrangle{Af}$	and $g_f(t) - \rho\M = e^{-t} \lrpar{f - \rho \M}$, for all $t\ge 0$, one obtains the following expression for $\phi_{2,0}$: for all $f\in\fSet$,
\begin{align}
	\phi_{2,0}(f)	
		&= \int_0^\infty \chi''_\rho \lrpar{\lrangle{Ag_f(t)},\xi}_{L^2_x}^2 dt + \int_0^\infty \chi'_\rho \lrpar{g_f(t)-\rho\M,A^2\xi + B\xi}_{L^2} dt \nonumber\\
		&= \int_0^\infty e^{-2t} \chi''_\rho \lrpar{\lrangle{Af},\xi}_{L^2_x}^2 dt + \int_0^\infty e^{-t} \chi'_\rho \lrpar{f-\rho\M,A^2\xi + B\xi}_{L^2} dt \nonumber\\
		&= \frac{1}{2} \chi''_\rho \lrpar{\lrangle{Af},\xi}_{L^2_x}^2 + \chi'_\rho \lrpar{f-\rho\M,A^2\xi + B\xi}_{L^2} \nonumber\\
		&= \frac{1}{2} \chi''_\rho \lrpar{f,A\xi}_{L^2}^2 + \chi'_\rho \lrpar{f-\rho\M,A^2\xi + B\xi}_{L^2}. \label{eq:def_phi_20}
\end{align}

On the other hand, the centering condition~\eqref{eq:centering_condition_Lm} is satisfied, and~\eqref{eq:Poisson_solution_Lm} gives the following definition for $\phi_{0,2}$: for all $\rho\in L_x^2$ and $\n\in E$, set
\begin{align}
	\phi_{0,2}(f,\n)
		&= \int_0^\infty \esp{\vartheta_{0,2}(\rho\M,m_\n(t)) - \int_E\vartheta_{0,2}(\rho\M,\cdot) d\nu} dt \nonumber\\
		&= - \int_0^\infty \esp{\chi'_\rho \lrpar{(\sigma(m_\n(t)) - \bar \sigma) \rho,\xi}_{L^2_x}} dt \nonumber\\
		&= - \chi'_\rho \lrpar{\rho,R_0(\n) \xi}_{L^2_x}. \label{eq:def_phi_02}
\end{align}

Using the expressions~\eqref{eq:def_phi_20} and~\eqref{eq:def_phi_02}, it is then straightforward to check that the estimates~\eqref{eq:control_phi_20} and~\eqref{eq:control_phi_02} are satisfied. This concludes the construction of the correctors $\phi_{2,0}$ and $\phi_{0,2}$.

\subsubsection{Construction of the corrector $\varphi_{1,2}$ using the condition~\eqref{eq:aim_corrector1}}

The last step in the construction of the correctors, is to define $\varphi_{1,2}$: owing to~\eqref{eq:aim_corrector1}, it is constructed as the solution of the Poisson equation $-\L_2\varphi_{1,2}=\L_1\varphi_{0,2}$. Using the expression~\eqref{eq:def_phi_02} of $\varphi_{0,2}$, one has
\begin{equation*}
	\L_1 \phi_{0,2}(f,\n) = \chi''_\rho \lrpar{\lrangle{Af},\xi}_{L^2_x} \lrpar{\rho,R_0(\n) \xi}_{L^2_x} + \chi'_\rho \lrpar{\lrangle{Af},R_0(\n) \xi}_{L^2_x},
\end{equation*}
for all $f\in\fSet$ and $\n\in E$.

Note that the centering condition~\eqref{eq:centering_condition_L2} is satisfied: for all $\rho\in L_x^2$ and $\n\in E$, one has
\[
\L_1 \phi_{0,2}(\rho\M,\n)=0,
\]
since $\lrangle{A\rho\M}=0$. Therefore, $\varphi_{1,2}$ is defined using~\eqref{eq:Poisson_solution_L2}: for all $f\in\fSet$ and $\n\in E$,
\[
	\phi_{1,2}(f,\n)=\int_0^\infty \L_1 \phi_{0,2}(g_f(t),\n) dt.
\]
Using the properties $\lrangle{g_f(t)} = \rho$ and $\lrangle{A g_f(t)} = e^{-t} \lrangle{Af}$ for all $t\ge 0$, one obtains the following expression for $\varphi_{1,2}$: for all $f\in\fSet$ and $\n\in E$,
\begin{align}
	\phi_{1,2}(f,\n)		
		&= \int_0^\infty e^{-t} \lrpar{\chi''_\rho \lrpar{\lrangle{Af},\xi}_{L^2_x} \lrpar{\rho,R_0(\n) \xi}_{L^2_x} + \chi'_\rho \lrpar{\lrangle{Af},R_0(\n) \xi}_{L^2_x}} dt \nonumber\\
		&= \chi''_\rho \lrpar{\lrangle{Af},\xi}_{L^2_x} \lrpar{\rho,R_0(\n) \xi}_{L^2_x} + \chi'_\rho \lrpar{\lrangle{Af},R_0(\n) \xi}_{L^2_x} \nonumber\\
		&= - \chi''_\rho \lrpar{f,A\xi}_{L^2} \lrpar{\rho,R_0(\n) \xi}_{L^2_x} - \chi'_\rho \lrpar{f,A(R_0(\n) \xi)}_{L^2}. \label{eq:def_phi_12}
\end{align}

Using the expression~\eqref{eq:def_phi_12}, it is then straightforward to check that the estimate~\eqref{eq:control_phi_12} is satisfied. This concludes the construction of the corrector $\phi_{1,2}$.

\subsubsection{Verification of the conditions~\eqref{eq:aim_corrector23_epsilon}--\eqref{eq:aim_corrector23_delta}}

Note that $\varphi^{\epsilon,\delta}$ defined by~\eqref{eq:def_phi_epsilon} belongs to the class of functions $\Theta$ given in Definition~\ref{def:Theta}. It thus only remains to check that the conditions~\eqref{eq:aim_corrector23_epsilon} and~\eqref{eq:aim_corrector23_delta} hold. This is done using the following expressions: for all $f\in\fSet$ and $\n\in E$, one has
\begin{align*}
	\L_0 \phi_{1,0}(f,\ell) &= - \chi''_\rho \lrpar{\sigma(\ell) \rho + \lrangle{Bf},\xi}_{L^2_x} \lrpar{f,A\xi}_{L^2} - \chi'_\rho \lrpar{\sigma(\ell) f,A\xi}_{L^2},
\end{align*}
\begin{align*}
	\L_1 \phi_{2,0}(f,\ell)
		&= \frac{1}{2} \chi'''_\rho \lrpar{f,A\xi}_{L^2}^3 + \chi''_\rho \lrpar{f,A\xi}_{L^2} \lrpar{f,A^2\xi}_{L^2}\\
		&\phantom{=}\ + \chi''_\rho \lrpar{f,A\xi}_{L^2} \lrpar{f-\rho\M,A^2\xi + B\xi}_{L^2} + \chi'_\rho \lrpar{f-\rho\M,A^3\xi + AB\xi}_{L^2},
\end{align*}
\begin{align*}
	\L_m \phi_{1,2}(f,\ell) &= \chi''_\rho \lrpar{f,A\xi}_{L^2} \lrpar{\rho,\sigma(\ell) \xi}_{L^2_x} + \chi'_\rho \lrpar{f,A(\sigma(\ell) \xi)}_{L^2},
\end{align*}
\begin{align*}
	\L_0 \phi_{2,0}(f,\ell)
		&= - \frac{1}{2} \chi'''_\rho \lrpar{\sigma(\ell) \rho + \lrangle{Bf},\xi}_{L^2_x} \lrpar{f,A\xi}_{L^2}^2 - \chi''_\rho \lrpar{f,A\xi}_{L^2} \lrpar{\sigma(\ell) f + Bf,A\xi}_{L^2}\\
		&\phantom{=}\ - \chi''_\rho \lrpar{\sigma(\ell) \rho + \lrangle{Bf},\xi}_{L^2_x} \lrpar{f-\rho\M,A^2\xi + B\xi}_{L^2}\\
		&\phantom{=}\ - \chi'_\rho \lrpar{\sigma(\ell) (f - \rho\M) + B(f - \rho\M),A^2\xi + B\xi}_{L^2},
\end{align*}
\begin{align*}
	\L_0 \phi_{0,2}(f,\ell) &= - \chi''_\rho \lrpar{\sigma(\ell) \rho + \lrangle{Bf},\xi}_{L^2_x} \lrpar{R_0(\ell) \rho,\xi}_{L^2_x} - \chi'_\rho \lrpar{\sigma(\ell) \rho + \lrangle{Bf},R_0(\ell) \xi}_{L^2_x},
\end{align*}
\begin{align*}
	\L_1 \phi_{1,2}(f,\ell)
		&= - \chi'''_\rho \lrpar{f,A\xi}_{L^2}^2 \lrpar{\rho,R_0(\ell) \xi}_{L^2_x} - \chi''_\rho \lrpar{f,A^2\xi}_{L^2_x} \lrpar{\rho,R_0(\ell) \xi}_{L^2_x}\\
		&\phantom{=}\ - \chi''_\rho \lrpar{f,A\xi}_{L^2} \lrpar{\rho,A(R_0(\ell) \xi)}_{L^2_x} - \chi''_\rho \lrpar{f,A\xi}_{L^2} \lrpar{f,A(R_0(\ell) \xi)}_{L^2}\\
		&\phantom{=}\ - \chi'_\rho \lrpar{f,A^2(R_0(\ell) \xi)}_{L^2},
\end{align*}
\begin{align*}
	\L_0 \phi_{1,2}(f,\ell)
		&= \chi''_\rho \lrpar{\sigma(\ell) \rho + \lrangle{Bf},\xi}_{L^2_x} \lrpar{f,A\xi}_{L^2} \lrpar{\rho,R_0(\ell) \xi}_{L^2_x} + \chi''_\rho \lrpar{\sigma(\ell) f Bf,A\xi}_{L^2_x} \lrpar{\rho,R_0(\ell) \xi}_{L^2_x}\\
		&\ + \chi''_\rho \lrpar{f,A\xi}_{L^2} \lrpar{\sigma(\ell) \rho + \lrangle{Bf},R_0(\ell) \xi}_{L^2_x} + \chi''_\rho \lrpar{\sigma(\ell) \rho + \lrangle{Bf},\xi}_{L^2_x} \lrpar{f,A(R_0(\ell) \xi)}_{L^2}\\
		&\ + \chi'_\rho \lrpar{\sigma(\ell) f + Bf,A(R_0(\ell) \xi)}_{L^2}.
\end{align*}
It is then straightforward to check that~\eqref{eq:aim_corrector23_epsilon} and~\eqref{eq:aim_corrector23_delta} hold. This concludes the proof of Proposition~\ref{prop:corrector}.

\section*{Acknowledgments}

The authors would like to thank Julien Vovelle for his remarks on an early version of this manuscript. The work of C.-E.~B. is partially supported by the following projects operated by the French National Research Agency: ADA (ANR-19-CE40-0019-02) and BORDS (ANR-16-CE40-0027-01).


\end{document}